\documentclass[11pt, reqno]{amsart}

\usepackage{fullpage}
\usepackage[mathscr]{eucal}
\usepackage[all]{xy}
\usepackage{mathrsfs}
\usepackage{xypic}
\usepackage{amsfonts}
\usepackage{amsmath}
\usepackage{amsthm}
\usepackage{amssymb}
\usepackage{latexsym}
\usepackage{eufrak}
\usepackage{fullpage}

\theoremstyle{plain}
\newtheorem{theorem}[equation]{Theorem}
\newtheorem{proposition}[equation]{Proposition}
\newtheorem{lemma}[equation]{Lemma}
\newtheorem{corollary}[equation]{Corollary}
\newtheorem{conjecture}[equation]{Conjecture}

\theoremstyle{definition}
\newtheorem{definition}[equation]{Definition}

\theoremstyle{remark}

\newtheorem{remark}[equation]{Remark}
\newtheorem{claim}[equation]{Claim}

\makeatletter
\renewcommand{\subsection}{\@startsection{subsection}{2}{0pt}{-3ex
plus -1ex minus -0.2ex}{-2mm plus -0pt minus
-2pt}{\normalfont\bfseries}} \makeatother

\numberwithin{equation}{subsection}

\newcommand{\Lmod}[1]{#1\text{-}{\mathsf{mod}}}

\newcommand{\hdot}{{\:\raisebox{3pt}{\text{\circle*{1.5}}}}}

\newcommand{\idot}{{\:\raisebox{1pt}{\text{\circle*{1.5}}}}}
\newcommand{\dd}[2]{\D_{#1,#2}}

\newcommand{\dul}{{\mathbf{DU}}}

\DeclareMathOperator{\Hom}{\mathrm{Hom}}

\DeclareMathOperator{\ssup}{{\mathrm{supp}}}

\DeclareMathOperator{\gr}{\mathtt{gr}}

\DeclareMathOperator{\Sym}{\mathrm{Sym}}

\newcommand{\supp}{\mathtt{spec}}
\DeclareMathOperator{\vol}{\mathrm{vol}}

\DeclareMathOperator{\Tr}{\mathrm{Tr}}
\DeclareMathOperator{\Lie}{\mathrm{Lie}}

\DeclareMathOperator{\ad}{\mathrm{ad}}

\newcommand{\dis}{\displaystyle}
\newcommand{\erem}{\hfill$\lozenge$\end{remark}}
\newcommand{\beq}{\begin{equation}\label}
\newcommand{\eeq}{\end{equation}}
\newcommand{\en}{\enspace}
\newcommand{\bo}{\mbox{$\bigotimes$}}
\renewcommand{\o}{\otimes}
\newcommand{\wh}{\widehat}
\renewcommand{\max}{{\mathrm{Max}}}

\newcommand{\wt}{\widetilde}
\DeclareMathOperator{\KZ}{\mathrm{KZ}}

\newcommand{\ccat}{{\scr C}_c(\X)}
\newcommand{\dcat}{{\mathbf{D}{\scr C}}_c(\X)}

\newcommand{\dcatchi}{{\mathbf{D}{\scr C}}_{\Th,c}(\X)}
\renewcommand{\th}{{\theta}}

\newcommand{\f}[1]{\mathfrak{#1}}
\newcommand{\scr}[1]{\mathscr{#1}}
\newcommand{\oper}{\operatorname}

\newcommand{\Om}{\Omega }

\DeclareMathOperator{\Spec}{\mathrm{Spec}}
\DeclareMathOperator{\pr}{pr}
\newcommand{\iso}{{\,\stackrel{_\sim}{\rightarrow}\,}}

\def\ccirc{{{}_{^{\,^\circ}}}}

\newcommand{\sminus}{\smallsetminus}
\renewcommand{\mid}{\enspace\big|\enspace}
\newcommand{\bop}{\mbox{$\bigoplus$}}
\newcommand{\too}{\,\,\longrightarrow\,\,}

\newcommand{\mto}{\mapsto}

\newcommand{\onto}{\twoheadrightarrow}
\newcommand{\cd}{\!\cdot\!}
\newcommand{\dash}{\mbox{-}}
\newcommand{\pb}{\noindent$\bullet\quad$\parbox[t]{140mm}}

\def\hp{\hphantom{x}}
\newcommand{\bmu}{{\boldsymbol{\mu}}}
\newcommand{\bnu}{{\boldsymbol{\nu}}}
\newcommand{\Don}{{{\mathbf{D}\mathsf{mon}}}}
\newcommand{\Mon}{{{\mathsf{Mon}}}}

\newcommand{\bff}{{\mathbf{f}}}
\newcommand{\fa}{{\mathfrak a}}

\renewcommand{\lll}{{\scr L}}

\newcommand{\vi}{${\sf {(i)}}\;$}
\newcommand{\vii}{${\sf {(ii)}}\;$}
\newcommand{\viii}{${\sf {(iii)}}\;$}
 
\newcommand{\syn}{{{\mathbb{S}}_n}}
\newcommand{\g}{{\mathfrak{g}}}

\newcommand{\hc}{{\mathsf{HC}}}
\newcommand{\chh}{{\mathsf{ch}}}
\newcommand{\hhc}{{\mathsf{hc}}}

\def\npb{\noindent$\bullet\quad$\parbox[t]{155mm}}

\newcommand{\dxc}{{\D_{\BX}^c}}

\newcommand{\Ga}{\Gamma }
\newcommand{\yy}{{\mathfrak B}}

\newcommand{\om}{{\varpi}}
\newcommand{\waff}{{W_{^{\oper{aff}}}}}

\newcommand{\inv}{^{-1}}
\newcommand{\oo}{{\mathcal O}}

\renewcommand{\part}{{\f P}}

\newcommand{\sv}{{\mathfrak{s}\mathfrak{l}}}

\newcommand{\arr}{\overset{{\,}_\to}}

\newcommand{\BT}{{\mathbb T}}
\newcommand{\bt}{{\mathfrak t}}

\def\C{{\mathbb{C}}}
\def\V{{\f G}}

\def\sln{{\mathfrak{s}\mathfrak{l}}_n(\C)}

\newcommand{\hh}{{\mathsf{H}}}

\newcommand{\RR}{R^{\!}}
\newcommand{\eps}{\epsilon}

\newcommand{\hcc}{{\Xi}}

\newcommand{\sset}{\subset}

\newcommand{\opp}{{\operatorname{op}}}

\newcommand{\D}{{\scr D}}
\renewcommand{\P}{{\mathbb{P}}}
\newcommand{\BI}{{\mathbb I}}

\newcommand{\X}{{\f X}}
\newcommand{\Th}{{\Theta}}
\newcommand{\into}{{\,\hookrightarrow\,}}

\newcommand{\Ug}{{\mathcal{U}\f g}}

\newcommand{\Z}{{\mathbb{Z}}}
\newcommand{\La}{\Lambda}
\newcommand{\LLa}{{\mathsf{\Lambda}}}

\newcommand{\Id}{\operatorname{Id}}

\newcommand{\FZ}{{\mathfrak{Z}}}

\newcommand{\xreg}{{\X^{\oper{reg}}}}

\newcommand{\treg}{{\BT^{\oper{reg}}}}
\newcommand{\baf}{{B^{\oper{aff}}_n}}

\newcommand{\hec}{{\scr H}_q }

\newcommand{\la}{{\lambda}}

\newcommand{\TV}{{\check\BT}}

\newcommand{\BH}{{\mathbb H}}
\newcommand{\sH}{{\mathsf H}}

\newcommand{\CE}{{\mathcal E}}
\newcommand{\CR}{{\mathcal R}}

\newcommand{\CB}{{\mathcal B}}
\newcommand{\CF}{{\mathcal F}}
\newcommand{\CG}{{\mathcal G}}
\newcommand{\CP}{{\mathcal P}}

\newcommand{\CJ}{{\mathcal V}}

\newcommand{\W}{{Z}}
\newcommand{\LL}{{\scr L}}

\newcommand{\CM}{{\mathcal M}}
\newcommand{\CN}{{\mathcal N}}
\newcommand{\CO}{{\mathcal O}}
\newcommand{\CU}{{\mathcal U}}

\newcommand{\CV}{{\mathcal V}}

\newcommand{\BC}{{\mathbb C}}
\newcommand{\BD}{{\mathbb D}}

\newcommand{\BX}{{\mathfrak X}}

\newcommand{\tB}{{\widetilde{\mathcal B}}}

\newcommand{\ch}{{\mathsf{CH}}}
\newcommand{\HC}{{\mathsf{HC}}}

\newcommand{\lr}{{\mathsf{lr}}}

\newcommand{\fH}{{\mathfrak H}}

\newcommand{\module}{{\operatorname{mod}}}
\newcommand{\Top}{{\operatorname{top}}}

\newcommand{\UU}{{\mathcal U}}
\newcommand{\BU}{{\mathbf U}}

\newcommand{\Ut}{{\mathcal U}\bt }

\begin{document}

\centerline{\Large{\textbf{On mirabolic D-modules}}}
\vskip 5mm

\centerline{\sc{Michael Finkelberg and Victor Ginzburg}}
\vskip 7mm
\centerline{\em To the memory of Israel Moiseevich Gelfand}
\vskip 7mm

\hskip 15mm\parbox[t]{130mm}{{\sc Abstract}.\
\footnotesize{Let  an algebraic group $G$ act
on  $X$, a connected algebraic manifold, with finitely many orbits.
For any Harish-Chandra pair $(\D,G)$ where
$\D$ is a sheaf  of twisted differential
operators on $X$, we form  a left ideal $\D{}^{\,}\g\sset\D$  generated
by the  Lie algebra $\g=\Lie G$.  Then,  $\D/\D{}^{\,}\g$ is a holonomic
$\D$-module, and its restriction to a unique  Zariski open
dense $G$-orbit in $X$ is a $G$-equivariant local system. We prove a 
criterion saying that the $\D$-module
$\D/\D{}^{\,}\g$ is isomorphic, under certain (quite restrictive) conditions,
to a direct image of that local system to $X$.
We apply this criterion in the special case of the group
$G=SL_n$ acting diagonally on $X=\CB\times\CB\times\P^{n-1},$ where
$\CB$ denotes the flag manifold for~$SL_n$.\newline
\hphantom{x}\ We further relate $\D$-modules on $\CB\times\CB\times\P^{n-1}$ 
to  $\D$-modules on the
cartesian product $SL_n\times\P^{n-1}$ via a pair $(\ch, \hc),$ of adjoint functors 
analogous to those used in Lusztig's theory of character sheaves.
A second important result of the paper provides an explicit
description of these functors, showing that the functor
$\hc$ gives an {\em exact} functor on the abelian
category of {\em mirabolic} ~$\D$-modules.}}

\bigskip

\centerline{\sf Table of Contents}
\vskip -1mm

$\hspace{20mm}$ {\footnotesize \parbox[t]{115mm}{
\hp${}_{\!}$\hp1.{ $\;\,$} {\tt Introduction}\newline
\hp2.{ $\;\,$} {\tt A result on holonomic $D$-modules}\newline
\hp3.{ $\;\,$} {\tt A triple flag variety}\newline
\hp4.{ $\;\,$} {\tt Mirabolic $D$-modules}\newline
\hp5.{ $\;\,$} {\tt Mirabolic Harish-Chandra $D$-module}\newline
\hp6.{ $\;\,$} {\tt Further properties of the  mirabolic Harish-Chandra
$D$-module}
}}

\section{Introduction} 
\subsection{} Let $X$ be
a smooth connected algebraic variety and 
 fix  $(\D,G)$, a Harish-Chandra pair  on $X$ in the
sense of \cite{BB3}, \S1.8.3. Thus, $G$ is
  an algebraic group  acting on $X,$ and
 $\D$ is a sheaf of (twisted) differential operators
on $ X$. One also has
a Lie algebra map
$\g:=\Lie G\to \D$ and  $\D{}^{\,}\g\sset\D$, the left
ideal generated by the image of this map.

Assume now that  $\dim X=\dim G$ and that $G$ is a reductive group
acting on $X$ with
finitely
many orbits. Let $U$ be a  unique
Zariski open dense
$G$-orbit in $X$.
Then, the quotient $\CJ:=\D/\D{}^{\,}\g$ is a holonomic $\D$-module on $ X,$
and $\CJ^\circ:=j^*\CJ$, the restriction of $\CJ$ via  the open
imbedding
 $j: U\into  X$,
is  a $G$-equivariant (twisted) local
system on $U$. 
The main result of section 2 (Theorem \ref{bthm})
 says that the $\D$-module
$\D/\D{}^{\,}\g$ is isomorphic, under certain  conditions
on the moment map $T^*X\to \g^*$ and certain bounds
on the roots of an associated $b$-function, cf. Remark \ref{masaki},
to either $j_*\CJ^\circ$ or $j_!\CJ^\circ$,
a direct image of of the local system $\CJ^\circ$ to $X$.

We are mostly interested in the special case where $G=SL_n$, 
and  $X=\CB\times\CB\times \P^{n-1}$.
Here, $\CB$ denotes the flag variety of the group $SL_n$,
so $X$ is a sort of `triple flag
manifold'.
It turns out  that the group $G$
 acts diagonally in   $X$ with
{\em finitely many}  orbits  (all cases of such an `unusual' phenomena
have been classified in \cite{MWZ}). Moreover, the 
 theorem from \S 2 applies in this case.

Our interest in $\D$-modules on  $X=\CB\times\CB\times \P^{n-1}$
is due to their close connection with {\em mirabolic character
$\D$-modules}.
\footnote{The name `mirabolic' comes from the combination
`miraculous parabolic', which is the parabolic 
subgroup $P\sset G$ that fixes a point in $\P^{n-1}$.
This parabolic was first considered in the
work by Gelfand and Kazhdan \cite{GK}, and it has some very 
 specific features. In our situation, considering
$G$-equivariant  $\D$-modules on the space $G\times \P^{n-1}$
is equivalent, essentially, to considering
$P$-equivariant  $\D$-modules on $G$. For this reason, we
use the name `mirabolic'.}
These are certain  $\D$-modules on the space $G\times \P^{n-1}$,
introduced in \cite{GG}, and studied further in \cite{FG}.
The category of mirabolic  $\D$-modules plays an important role in the study of category
$\CO$ for a rational Cherednik algebra, cf. \cite{GG}, \cite{FG},
and references therein.

One of our key  observations is 
 that (a derived version of) the category of mirabolic
$\D$-modules
is related to a (derived) category of  $\D$-modules on the space $\CB\times\CB\times \P^{n-1}$
via a pair $(\ch,\ \hc)$ of adjoint functors.
Here, the {\em character functor} $\ch$ is a `mirabolic counterpart' of
a similar functor used by Lusztig \cite{Lu} in his theory of {\em character
sheaves}
on the group $G$. The functor $\hc$ is a `mirabolic
counterpart' of the  {\em Harish-Chandra} functor considered in \cite{G}.

The second important result of the paper is 
a mirabolic analogue of a result from \cite{BFO}. 
Specifically, part (ii) of our Theorem \ref{hk}
 says
that the functor $\hc$ gives an exact functor
from the {\em abelian} category  of mirabolic
$\D$-modules to an appropriately defined abelian category of 
Harish-Chandra modules.

There is one especially important  mirabolic
$\D$-module, called  {\em Harish-Chandra  $\D$-module}.
To define it, write 
$\D^c$ for the sheaf of twisted differential operators
on $G\times\P^{n-1}$ with twist $c\in\C$ along the second factor.
Let $\Delta: \Lie G\to \D^c$ be the
map induced by the differential of the $G$-diagonal action on $G\times\P^{n-1},$ via
$g:\ (g',l)\mto (gg'g\inv, g(l)),$ and put $\g:=\Delta(\Lie G)\sset \D^c.$
Further, let $\D(G)^{G\times G}$ be the algebra of $G$-{\em bi-invariant}
differential operators on $G$. Thus $\D(G)^{G\times G}$ is a commutative algebra isomorphic
to
the center of $\UU(\Lie G)$, the universal enveloping algebra of the Lie
algebra $\Lie G$. 

Following \cite{GG}, \S 7.4,
for any maximal ideal
$\FZ_\th\sset \D(G)^{G\times G}$, one 
defines an associated {\em mirabolic
Harish-Chandra  $\D$-module} to be, cf. Definition \ref{hcmod} for more details,
\beq{hcmod_intro}
\CG^{\th,c}:=\D^c\big/(\D^c\,\g+
\D^c\, (\FZ_\th\o 1)).
\eeq

This is a $G$-equivariant, holonomic $\D$-module on $G\times \P^{n-1}$,
analogous to a $\D$-module on the group $G$ itself studied by Hotta-Kashiwara in \cite{HK1},
\cite{HK}.
In the last section of the paper, we use the functor $\ch$ and our general result 
from
section 2 to obtain a purely geometric construction of the
perverse sheaf that corresponds to  the mirabolic Harish-Chandra
$\D$-module
via the Riemann-Hilbert correspondence. We also study  Verdier duality
for the mirabolic 
Harish-Chandra
$\D$-module $\CG^{\th,c}$ and give a conjectural description of
the restriction of  $\CG^{\th,c}$ to an open dense subset
of $G\times\P^{n-1}$ (see Conjecture \ref{factor}).

\subsection{Acknowledgements}{\footnotesize
We are very much indebted  to  Roman Bezrukavnikov for communicating to
us his ideas, related to the results of
\cite{BFO}, before they were made public.
We are grateful to Sasha Beilinson for his kind
explanations concerning monodromic $\D$-modules.
We also thank an anonymous referee for his helpful comments. 
M.F. is  grateful to IAS, to the University of Chicago, 
and to Indiana University
at Bloomington for the hospitality and support. The work of M.F. was partially
supported by the RFBR grant 09-01-00242 and the Science Foundation of the
SU-HSE awards No.09-08-0008 and No.09-09-0009.
The work of V.G. was  partially supported by the NSF grant DMS-0601050.}

\section{A result on holonomic $D$-modules}
\subsection{}\label{TDO}
For any smooth variety $X$, write $\Om^p_X$ for the sheaf
of algebraic differential $p$-forms on $X$,
and  $\omega_X$ for the canonical line bundle of top degree differential
forms. Further, let $(\Omega_X^p)_\text{closed}
\sset\Om^p_X$ denote
the subsheaf of closed $p$-forms. 
Abusing the notation, we  write $\omega\in H^2(X,
(\Omega^1_X)_\text{closed})$
for the  Chern class of the canonical bundle $\omega_X$.

We refer the reader to \cite{BB3} for  generalities on
twisted  differential operators (TDO), and to ~\cite{K}, ~\S2, for a
survey of basic
functors on twisted $\D$-modules.

We say that  $\D$, a sheaf of TDO on $X$, is {\em locally trivial}
if, locally in \'etale topology, one
has an isomorphism $\D\cong\D_X$, where $\D_X$ stands for
the sheaf of (nontwisted)  algebraic differential operators  on $X$.
It is known that the sheaves of algebraic locally trivial
TDO are parametrized (up to
isomorphism) by elements of the group 
$H^1_{{\text{\'et}}}(X, (\Omega^1_X)_\text{closed})$, \cite{BB3}.

{\em All TDO considered in this paper are assumed,
without further notice, to be
locally trivial}.
Given  $\D_\chi$, the sheaf of twisted differential
operators  on $X$ associated with a class
$\chi\in H^2(X, (\Omega^1_X)_\text{closed})$, we write $\D_{-\chi}$ for the sheaf of differential
operators with the opposite twisting.

Given an associative algebra, resp. sheaf of 
algebras, $A$, write $A^\opp$ for the
opposite algebra. 
Thus, for any TDO $\D_\chi$,
one has the sheaf $\D_\chi^\opp:=(\D_\chi)^\opp$.
There is a canonical isomorphism of TDO's
$\D_\chi^\opp\cong\D_{\omega-\chi}$,
see ~\cite{K}, 2.7.

We will use the notion of {\em regular singularities}
for modules over (locally trivial) TDO.
Let $\D$ be such a  TDO on $X$ and
let $j: U\into X$ be an open imbedding.
Let $M$ be a holonomic $j^*\D$-module. Locally in
\'etale topology, one may view $j_*M$ as
a module over $\D_X$, the sheaf of nontwisted differential
operators on $X$.
We say that $M$ is {\em regular
at a point} $x\in X$ if the following holds:\
{\em For any smooth  curve $C\sset X$ containing $x$,
the restriction of $j_*M$ to $C$ is either supported at
the point $x$, or is smooth
at $x$, or else it has a regular singularity at $x$.}

\subsection{}
Let  a connected algebraic group $G$  act on $X$, a smooth 
connected variety.
 Let $U\sset X$ be a (unique) Zariski open dense $G$-orbit,
and assume
$X\sminus U$ is a  hypersurface,
not necessarily irreducible, in general.

Let ${\scr L}$  be a  $G$-linearization
of the line bundle $\CO_X(Y)$ where
$Y$ is a divisor with support equal to that
hypersurface. Thus, $\LL$ is a $G$-equivariant line bundle
 on $X$ and there exists a  regular section
 $s$, of $\LL,$ such that $U=X\sminus s\inv(0).$
The following result is well known.

\begin{lemma}\label{lemphi} There exists a   group homomorphism  $\phi: G\to\C^\times$
and a  $\phi$-semi-invariant regular section
 $s\in\Gamma(X,\LL),$ 
such that 
\begin{equation}\label{phi}
U=X\sminus  s\inv(0)\en\;\operatorname{and,\en we\en have}
\quad g^*(s)=\phi(g)\cdot s,\qquad\forall g\in G.
\end{equation}
\end{lemma}
\begin{proof} Let $s\in\Gamma(X,\LL)$ be any
 regular section  such that $U=X\sminus s\inv(0).$
Then, for any $g\in G$, we have
$(g^*s)\cdot s\inv\in\CO(X)^\times,$ is an invertible
regular function on $X$. The assignment
$g\mto (g^*s)\cdot s\inv$ gives an 
algebraic cocycle that defines a
cohomology class in $H^1_\text{alg}(G,\,\CO(X)^\times).$

Now, since $G$ is connected, from \cite{KKV},
Proposition 5.1, we deduce that any cohomology class in
  $H^1_\text{alg}(G,\,\CO(X)^\times),$ may be represented by the class of
 a character $\phi: G\to\C^\times.$
Therefore, there exists  a character $\phi$ and an invertible
function $f\in \CO(X)^\times$ such that we have
$(g^*s)\cdot s\inv=\phi(g)\cdot (g^*f)\cdot f\inv,$
for any $g\in G$. It follows that
$f\inv\cdot s$ is a $\phi$-semi-invariant section
with the required properties.
\end{proof}

 For any $x\in X,$ we let $G_x\sset G$ denote the isotropy
group of the point $x$. Write  $\g:=\Lie G$, resp. $\g_x:=\Lie G_x,\,\forall x\in X.$
The  $G$-equivariant structure on $\LL$
gives rise,
for any $x\in X,$ to  a group homomorphism $\chi_{\LL,x}:
G_x\to\C^\times$,
induced by the $G_x$-action in the fiber of $\LL$ at $x$.
Observe that, for any $x\in U$, we have $\phi|_{_{G_x}}=\chi_{\LL,x}.$
We keep the same notation, $\phi$ and $\chi_{\LL,x}$, for
  Lie algebra homomorphisms
corresponding to the group characters introduced above.

\subsection{} \label{psi}
Let $(\D, G)$ be a Harish-Chandra algebra on $X$, 
in the sense of \cite{BB3}, \S1.8.3; Put $\g:=\Lie G,$ and
let $\g\to \D,$
$u\mto\arr{u}$ be the corresponding
Lie algebra  homomorphism.

Given a $G$-equivariant
line bundle  $\LL$ on $X$ we introduce the following notation
$$\D_k:=\LL^{\o k}\o_{\CO_X}\D\o_{\CO_X}\LL^{\o (-k)},
\quad\CF(k):=\LL^{\o k}\o_{\CO_X}\CF,
\quad k\in\Z,
$$
for any $\CO_X$-module $\CF$.
It is clear that, for any $k\in\Z$, the pair $(\D_k, G)$ has a
structure of Harish-Chandra
algebra as well, in particular,
there is  a Lie algebra map $\g\to\D_k,\, u\mto\arr{u}.$

One also has
$(\D_k)^\opp$,
an opposite TDO, and there is a
natural isomorphism $(\D_k)^\opp\cong (\D^\opp)_{-k}$.
From now on, we will write $\D^\opp_{-k}:=(\D^\opp)_{-k}$.
For any  $\D$-module $\CF$, the sheaf
$\CF(k)$ has a natural $\D_k$-module structure.
Thus, a right $\D_k$-module is the same thing
as a left $(\D^\opp_{-k})$-module.

Given a 1-dimensional character $\psi: \g \to\C$,
let $\g^\psi\sset \D$ denote the image of $\g$
under the  Lie algebra morphism $u\mto\arr{u}-\psi(u)\cdot 1.$
Let $\D{}^{\,}\g^\psi\sset\D$ be the left ideal generated by
$\g^\psi$, 
and let  $\D/\D{}^{\,}\g^\psi$, be the corresponding left $\D$-module on~$X$.

Let  $U\sset X$ be a Zariski open dense $G$-orbit and
write $\jmath: U\into X$ for the  imbedding.
Fix  $k\in \Z$, and let
 $\jmath^*(\D_k/\D_k{}^{\,}\g^\psi),$ be
the restriction of
the $\D_k$-module $\D_k/\D_k{}^{\,}\g^\psi$ to $U$. It is clear that
the $G$-action on $\D_k$ makes $\jmath^*(\D_k/\D_k{}^{\,}\g^\psi)$
a $G$-equivariant locally free coherent
$\CO_U$-module. 
Furthermore, if  $\psi|_{\g_x}\ne0$ for some $x\in U$ then, one must
have $\jmath^*(\D_k/\D_k{}^{\,}\g^\psi)=0.$
Otherwise, $\jmath^*(\D_k/\D_k{}^{\,}\g^\psi)$ is a
line
bundle on $U$.

Let $\delta: \g\to\C,\, u\mto \Tr\ad u$ be the {\em modular character}
of the Lie algebra $\g$.

\begin{theorem}\label{bthm} Let a  connected {\em reductive}
 group $G$ act on a smooth variety $X$, such  that  $\dim X=\dim G$,
  with finitely many orbits. Let $(\D,G)$ be a
Harish-Chandra pair, where $\D$ is a
{\em locally trivial} TDO on $X$.
Then,   for any character
  $\psi:\g\to\C$ and any
$k\in\Z$, one has

\vi The   $\D_k$-module $\,\D_k/\D_k{}^{\,}\g^\psi$
is   holonomic  and regular at any point $x\in X$, cf. \S\ref{TDO};
furthermore,
there is a natural
isomorphism
$$R{{\scr H}_{\!}om}_{\D_k}(\D_k/\D_k{}^{\,}\g^\psi,\D_k)\cong
\D^\opp_{-k}/\D^\opp_{-k}{}^{\,}\g^{\delta-\psi}[-\dim X].
$$

\vii Let $U\sset X$ be a Zariski open dense $G$-orbit in $X$,
 let $\phi:G\to\C^\times$ be as in Lemma \eqref{lemphi} and
assume, in addition, that 
\begin{equation}\label{match}
\chi_{\LL,x}\neq \phi|_{\g_x},
\qquad\forall x\in X\sminus U.
\end{equation}

Then, one has canonical isomorphisms:
\begin{eqnarray}\label{concl_thm}
&\D_{k}/\D_{k}{}^{\,}\g^{\psi+{k}\cdot\phi}\ \iso \ \jmath_*\jmath^*(\D_{k}/\D_{k}
{}^{\,}\g^{\psi+k\cdot\phi}),\quad&\oper{for}\en k\ll0;\\
&\jmath_!\jmath^*(\D_{k}/\D_k{}^{\,}\g^{\psi+k\cdot\phi})\ \iso\ \D_{k}/\D_{k}{}^{\,}\g^{\psi+k\cdot\phi}
,\quad&\oper{for}\en k\gg0.\nonumber
\end{eqnarray}
\end{theorem}

It would be interesting to apply this theorem to
spherical varieties, and also to some prehomogeneous
vector spaces.


The rest of this section is devoted to the proof of Theorem \ref{bthm}.

\subsection{Geometry of the moment map.}
\label{m-sec}
 In this section, we prove a few geometric
results which will enable us to reduce the proof of Theorem \ref{bthm}
to  Proposition \ref{thm_cycle} below.

Given an arbitrary smooth $G$-variety $X$ and
 a $G$-equivariant line bundle $\LL$, on $X$,
write  $\LL^*$ for the dual  line bundle.
Let $L\to X$ be  a principal $\C^\times$-bundle obtained by removing the
zero section in the total space of the line bundle $\LL^*$.
The  $\C^\times$-action on $L$  commutes with the
natural $G$-action, hence makes $L$ a  $G\times\C^\times$-variety.

Let $T^*L$ be the total space of the cotangent bundle on
$L$. One has a natural
 Hamiltonian  $G\times\C^\times$-action  on $T^*L,$
with  moment map $T^*L \to \g^*\times(\Lie\C^\times)^*$.
Observe that the action of the second factor $\C^\times\sset G\times\C^\times$
on $T^*L$ is free.
 Thus, the quotient $\W:=(T^*L)/\C^\times$ is a smooth variety,
and the moment map above descends to the quotient. The resulting map may be written
in the form of a cartesian product of two maps
$\bmu\times\bnu: \W \to \g^*\times\C$.

It is well known that
 the map $\bnu: \W\to\C$ is a  smooth morphism,
and for each $a\in\C$, the   fiber $\bnu\inv(a)$ is a
 symplectic manifold. Furthermore, for $a=0$, one has a canonical
isomorphism $\bnu\inv(0)\cong T^*X,$ such that
the restriction of the map $\bmu$ to $\bnu\inv(0)$
may be identified with 
$$\mu=\bmu|_{\bnu\inv(0)}:\ \bnu\inv(0)=T^*X\to\g^*,$$
the moment map for the  natural Hamiltonian
 $G$-action on $T^*X$.

There is also a (non-Hamiltonian) $\C^\times$-action on 
$T^*L$,  by dilations along the fibers.
This action, to be referred to as  `dot-action',
 descends to a  $\C^\times$-action $\C^\times\ni a: z\mto a\cdot z,$ on 
$\W=(T^*L)/\C^\times$, the quotient of
$T^*L$ by the Hamiltonian   $\C^\times$-action. There is also
a  `dot-action' on  $\g^*\times\C$, defined as the
 $\C^\times$-diagonal action  by dilations. The dot-actions of $\C^\times$
on $\W$ and on $\g^*\times\C$ are both
{\em contractions}, and the map
$\bmu\times\bnu: \W\to\g^*\times\C$ is clearly dot-equivariant.

Let $s\in\Gamma(X,\LL)$ be a $\phi$-semi-invariant section,
so that \eqref{phi} holds. Put
$U:=X\sminus s\inv(0)$, and let $\wt U$ denote the preimage of $U$ in
$L$, the total space of $\LL^*$.

The
section $s$ may (and will) be viewed as a regular
function on $L$. The graph of the 
 closed 1-form $d\log s=s\inv\cdot ds$ may be
viewed as a section of $T^*L$
over $\wt U$. 
Let $\La\sset \W$ be the image 
of the graph of $d\log s$ under the projection $T^*L\onto (T^*L)/\C^\times=\W.$

\begin{lemma}\label{cond_lem} Let 
$U \sset X$ be a Zariski open
$G$-orbit such that \eqref{phi} holds. Then,

\vi The set $\Lambda$ is a smooth, closed Lagrangian
submanifold of the symplectic manifold $\bnu\inv(1)$. Furthermore, we have
$$\Lambda\sset\bmu\inv(\phi)\cap\bnu\inv(1).
$$

\vii Assume also  that the following conditions \eqref{item1}-\eqref{item3}  hold:
\begin{align}
&
\text{$\bullet\en$ The intersection $\bmu\inv(\phi)\cap\bnu\inv(1)$ 
is reduced at any point of $\Lambda$;}
\label{item1}\\
&
\text{$\bullet\en$ The group $G$ acts on $X$ 
with finitely many orbits.}\label{item3}
\end{align}
Then, 
$\Lambda$ is  an irreducible component of the scheme
$\bmu\inv(\phi)\cap\bnu\inv(1)$.
\vskip 3pt

\viii If, in addition, condition~\eqref{match} holds
then $\bmu^{-1}(\phi)\cap\bnu^{-1}(1)=\La,$ is an irreducible scheme.
\end{lemma}

\begin{proof} To prove (i) we observe that the
function $s$ on $L$ associated with any
 section  of $\LL$ is a degree 1 homogeneous function,
that is, we have $s(a\cdot \ell)=a\cdot s(\ell),$
for any $a\in\C^\times,\,\ell\in L$. For such a function,
the graph of $d\log s$ is  stable under the Hamiltonian
$\C^\times$-action on 
$T^*L$ and is contained in the fiber of the moment map over
the subset
$\g^*\times \{1\}\sset\g^*\times \C$.
Trivializing the line bundle $\LL$ locally,
one sees  (cf. \cite[\S5]{K}) that this graph is
 a smooth and
closed  subvariety of $T^*L$.
It follows that $\Lambda$, the quotient  of the graph 
by a free $\C^\times$-action, is  a smooth and
closed  subvariety of  $\bnu\inv(1)$.

Observe further, for a $\phi$-semi-invariant section
of $\LL$, one has $\arr{u}(s)=\phi(u)\cdot s,\,
\forall u\in\g$. This equation yields 
$\Lambda\sset \bmu\inv(\phi),$
and part (i) follows.

To prove (ii), let $S$ denote the {\em finite} set of all $G$-orbits in $X$.
It is well known that,  set-theoretically, we have that $
\mu\inv(0)=
\bigsqcup_{O\in S}\ T^*_{O}X,$ is 
the union  of the conormal bundles to $G$-orbits. Thus, we compute
\begin{equation}\label{dimension}
\dim\big(\bmu\inv(0)\cap\bnu\inv(0)\big)=
\dim\mu\inv(0)=
\dim\big(\bigsqcup_{O\in S} T^*_{O}X\big)=
\mbox{$\frac{1}{2}$}\cd\dim T^*X=\dim X.
\end{equation}

Moreover, the above shows that the dimension of
{\em any} irreducible component of the intersection
$\bmu^{-1}(0)\cap\bnu^{-1}(0)$ equals $\dim X$.
Further, the dot-action being a contraction,
we deduce an inequality
$\dim[\bmu^{-1}(\phi)\cap\bnu^{-1}(1)]\leq\dim[\bmu^{-1}(0)\cap\bnu^{-1}(0)]$.
Thus, using part (i), we obtain
$$\dim \Lambda\leq\dim[\bmu^{-1}(\phi)\cap\bnu^{-1}(1)]\leq
\dim[\bmu^{-1}(0)\cap\bnu^{-1}(0)]
=\dim X=\dim\Lambda.
$$
Thus, we have $\dim  \Lambda=\dim[\bmu^{-1}(1)\cap\bnu^{-1}(1)]$,
and part (ii) of the lemma follows from \eqref{item1}.

To prove (iii), consider the following diagram
\begin{equation}\label{diagram}
\xymatrix{
\bmu\inv(\C\cd\phi)\cap\bnu\inv(1)\;\ar[d]_<>(.5){\bmu}
\ar@{^{(}->}[rr]&&\,\bnu\inv(1)\ar@{->>}^<>(.5){\pr_{_X}}[drr]\ar[d]_<>(.5){\bmu}&&\\
\C\cd\phi\;\ar@{^{(}->}[rr]
&&\,\g^*&&\quad X\supset U.
}
\end{equation}
Part (ii)  of the lemma implies that we have
$\pr_X\inv(U)\cap\bmu^{-1}(\phi)\cap\bnu^{-1}(1)=\Lambda.$

We leave to the reader to verify that,
 for any point $x\in \pr_X(\bmu\inv(\phi)\cap\bnu\inv(1)),$
one must have $\chi_{\LL,x}= \phi|_{\g_x}$.
Hence, condition \eqref{match} insures that
$\pr_X(\bmu^{-1}(\phi)\cap\bnu^{-1}(1))\sset U$, and
(iii) is proved.
\end{proof}

\begin{lemma}
\label{muflat} 
Let the group
$G$ act on  $X$ with finitely many orbits. Then,
we have $\dim\mu\inv(0)=\dim X$. 
If, in addition,
 $\dim X=\dim G$, then the following holds:
\vskip 2pt

\vi Each of the two  maps $\mu: T^*X\to\g^*$ and  
$\bmu\times\bnu: \W\to\g^*\times\C$ is flat;

\vii Any fiber of the moment map $\mu$
is a complete intersection in $T^*X$;

\viii The two conditions of Lemma \ref{cond_lem}(ii) hold.

\end{lemma}
\begin{proof} The equation $\dim\mu\inv(0)=\dim X$ is clear from
\eqref{dimension}.

Observe next that each of the maps $\mu$ and $\bmu\times\bnu$,
is an
equivariant
morphism between smooth varieties with contracting
$\C^\times$-actions, the `dot-actions'. In such a case, 
the dimension of any fiber of the morphism
is less than or
equal to 
the dimension of the zero fiber. 
Thus,  any fiber of either
 $\mu$ or $\bmu\times\bnu$, has dimension
less than or
equal to $\dim X.$ Now, the assumption that
$\dim X=\dim G$ implies that the fiber dimension
is equal to $\dim T^*X-\dim \g^*.$
We conclude that each of the maps $\bmu\times\bnu$ and $\mu$ is flat.
This yields part (i). 

Part (ii) also follows,
since $T^*X$ and $\g^*$ are
smooth schemes and any fiber of a flat morphism of smooth schemes
is a complete intersection. 

We now prove (iii). Condition \eqref{item3} is clear.
To prove ~\eqref{item1}, we use diagram \eqref{diagram},
and set $\bmu^{-1}(\BC\cdot\phi)_U:=\pr_X\inv(U)\cap\bmu^{-1}(\BC\cdot\phi).$
This is an open subset in $\bmu^{-1}(\BC\cdot\phi)$.
Observe that, since the group $G$ acts transitively
 on
$U$  with finite stabilizers, it follows that
the scheme-theoretic intersection $\bmu^{-1}(0\cdot\phi)_U\cap\bnu^{-1}(0)$
is the {\em reduced} zero section $U\subset T^*U$.
Hence, the  general fiber
of the map $\bnu: \
\bmu^{-1}(\BC\cdot\phi)_U\to \BC\cdot\phi$ is reduced as well.
But {\em any} nonzero fiber of this map
may be viewed as `general', due to
the $\C^\times$-action. Thus, the fiber
$\bmu^{-1}(\BC\cdot\phi)_U\cap\bnu^{-1}(1)\supset\Lambda$ is reduced,
and ~\eqref{item1} is proved.
\end{proof}

\subsection{Reduction of the proof of  Theorem \ref{bthm}.}
We need to review some basic definitions concerning
$G$-monodromic $\D$-modules. 

Let $\psi:\g\to\C$ be a 1-dimensional character.
 First of all,
we introduce a flat connection on $\oo_G$,
a rank one trivial line bundle on $G$, defined by the formula
$\nabla_u(f):=u(f)-\psi(u)\cdot f$, for any
$f\in\oo_G$ and any left  invariant vector field
$u$ on $G$ identified with the corresponding element of 
the Lie algebra $\g$.
This connection makes the structure sheaf $\oo_G$
a $\D_G$-module, to be denoted $\oo^\psi_G$.

Now,
let $X$ be an arbitrary  smooth $G$-variety,
with the action map $a: G\times X\to X$.
In the setting of  \S\ref{psi}, 
let $(\D,G)$ be  a Harish-Chandra pair on $X$.
By \cite[\S1.4.5(ii)]{BB3}, one has a natural isomorphism
$a^*\D\cong \D_G\boxtimes\D$, of TDO.

Let $M$ be a $\D$-module on $X$. We say
that $M$ is $(G,\psi)$-{\em monodromic}
if $M$ is a {\em weakly} $G$-equivariant
$\D$-module (i.e. a `weak $(\D,G)$-module'
in the sense of \cite[\S1.8.5]{BB3}) and
there is an isomorphism $a^*M\cong \oo^\psi_G\boxtimes M$,
of $\D_G\boxtimes\D$-modules, such that the natural
cocycle condition holds.

The following result is an extention of
 \cite[VII, \S12.11]{Bo}, where a similar result
was proved for (strongly) $G$-equivariant $\D$-modules.

\begin{lemma}\label{mon} Let $G$, a reductive group,
 act on $X$ with finitely many orbits.
Let $(\D,G)$ be a Harish-Chandra pair, where $\D$ is
a locally trivial TDO on $X$.
Then, any  $(G,\psi)$-monodromic
$\D$-module is regular at every point $x\in X$.
\end{lemma}

\begin{proof} We follow the same strategy as in the proof of 
\cite[Theorem 11.6.1]{HTT},
cf. also \cite[VII, \S12.11]{Bo}. First of all, 
 we observe that   $\oo^\psi_G$ is a regular   $\D_G$-module
in the usual sense
(i.e. for any completion $j:G\into\overline{G}$,
the $\D_{\overline{G}}$-module $j_*\oo^\psi_G$ is regular  at
any point $g\in \overline{G}$),
 provided the group
$G$ is reductive. This is verified directly in the case where $G$ is a
complex torus;
the case of a general reductive group $G$ can be easily reduced
to the case of a torus.

Now, let $Y$ be an arbitrary smooth $G$-variety and let
 $(\D,G)$ be a Harish-Chandra pair, where $\D$ is
a locally trivial TDO on $Y$ and $G$ is a reductive group.
Let $\imath: X\into Y$ be an imbedding
of a $G$-stable smooth locally closed subvariety.
There is a well defined pull-back $\imath^*\D$, a TDO  on $X$,
cf. \cite[\S1.4]{BB3}.

We claim the following: {\em If 
$X$ is a finite union of $G$-orbits, then any  $(G,\psi)$-monodromic
$\imath^*\D$-module $M$, on $X$, is regular at any point
$y\in Y$} (this statement is vacuous unless $y$ is contained in
 the closure of $X$). 
In the special case where  $X=G/K$ is a single $G$-orbit,
the claim
is a straightforward consequence of the
regularity of the $\D_G$-module  $\oo^\psi_G$.

We now prove the claim  in the general case.
Let $O$ be a closed $G$-orbit in $X$.
Write $i: O\into X$ for the imbedding,
and write $a_O$, resp. $a_X$, for the $G$-action
morphism on $O$, resp. on $X$.
Let $i^!M$ be the (derived) restriction of $M$,
a $(G,\psi)$-monodromic
$\D$-module on $X$, to $O$. 
Thus,  $i^!M$ is a complex
and each cohomology group ${\scr H}^p(i^!M)$, of that complex,
is a holonomic $i^*\D$-module on $O$.
Furthermore, we observe that  ${\scr H}^p(i^!M)$
is  a $(G,\psi)$-monodromic
$i^*\D$-module. This follows easily by equating compositions of
derived restriction functors induced by the
following equal composite maps
$i\ccirc a_O=a_X\ccirc(\Id_G\times i):\
G\times O \to X.$

We deduce, using the result in the case of one orbit, that
 each cohomology group ${\scr H}^p(i^!M)$
is regular at any point $y\in Y$.
The proof of the claim is now
completed  by induction
on the number of $G$-orbits in $X$.
The argument is based
on a long exact sequence (the latter works
for locally trivial TDO similarly to
the usual case) in the same way as in the proof of
\cite[Theorem 11.6.1]{HTT}.

Finally, applying the claim in the case $X=Y$
yields the statement of the lemma.
\end{proof}
 
We now return to the setting of   Theorem \ref{bthm}.
It is immediate from definitions that
$\D_k/\D_k{}^{\,}\g^\psi$ is a $(G,\psi)$-monodromic
$\D_k$-module.
The number of $G$-orbits on $X$ being finite and the group $G$
being reductive,
we deduce from Lemma \ref{mon}
that $\D_k/\D_k{}^{\,}\g^\psi$ is a holonomic
$\D_k$-module which is regular at every  $x\in X$.
At this point,
it is clear that Theorem \ref{bthm}
follows from Lemma  \ref{cond_lem} and
 the following more general result.

\begin{proposition}\label{thm_cycle} Let $G$ act on $X$ with finitely
many orbits. Then,

\vi  If $\dim X=\dim G$ then, for any  character $\psi: \g\to\C$, 
there is a natural isomorphism
$$R{{\scr H}_{\!}om}_{\D_k}(\D_k/\D_k{}^{\,}\g^\psi,\D_k)\cong
\D^\opp_{-k}/\D^\opp_{-k}{}^{\,}\g^{\delta-\psi}[-\dim X],\qquad\forall k\in\Z.
$$

\vii Assume that conditions \eqref{item1}-\eqref{item3} as well as
~\eqref{match} hold and, moreover, that $\D$ is
a locally trivial TDO.
 Then,  for all sufficiently  negative integers $k\ll0$, 
the canonical map in \eqref{concl_thm} is an isomorphism.
\end{proposition}

Part (i) of the proposition is an immediate consequence
of Lemma \ref{muflat}, and of  Lemma \ref{duality_lem} of section
\ref{le}
below.
The proof  of
part (ii) is based on some results of Kashiwara and will be given in the following subsection.

\subsection{An application of the $b$-function.}
We begin with
 the following general result whose proof is
obtained by a standard application of the theory of $b$-functions,
cf. \cite{K1}. Let ${\mathsf X}$ be a manifold,
$f:\ {\mathsf X}\to\BC$ be a regular function, and 
${\mathsf U}:={\mathsf X}\sminus
f^{-1}(0)\stackrel{j}{\hookrightarrow}{\mathsf X}$.
Let $\D$ be a
 locally trivial TDO on ${\mathsf X}$, and
write  
$\D_{\mathsf U}$ for the restriction of  $\D$ to
the open set ${\mathsf U}$.  
\begin{lemma}\label{bfunction}
Let $\CE=\D_{\mathsf U}\cd e$ be a cyclic, 
holonomic $\D_{\mathsf U}$-module generated by
an element $e\in \CE$. Then, for any $k\ll0$, the element
$f^k\cd e$ is a generator for the $\D$-module
$j_*\CE,$ that is, we have $j_*\CE=\D\cd(f^k\cd e).$ \qed
\end{lemma}

We return now to the setup of Proposition  \ref{thm_cycle} and put
 $\CE_k:=\jmath^*(\D_k/\D_k{}^{\,}\g^{\psi+k\cdot\phi}),$ so $\CE=\CE_0$.
\begin{lemma} 
\label{zabyl}
For all $k\ll0,$ the canonical map
$\D_k/\D_k{}^{\,}\g^{\psi+k\cdot\phi}\to\jmath_{*}\CE_k$ 
is surjective.
\end{lemma}

\begin{proof}
By adjunction, there is a  canonical map
$$
\D/\D\g
\to\jmath_*\jmath^*(\D/\D\g)=:\jmath_*\CE.
$$
Let $e=e_0\in\jmath_*\CE$ be the image of the class of $1\in
\D$. Clearly, we have $\D_U\cd e=\CE$.

For any  $k\in\Z$, let $e_k:=s^k\cd e\in\jmath_*\CE_k $. 
Since $\D_U\cd e_0=\CE,$
 Lemma~\ref{bfunction} implies that
there exists $k_0<0$ such that, for all $k<k_0$, 
one has $\D_k\cd
e_k=\jmath_*\CE_k .$

The element $e_k$ is clearly annihilated
by the action of the ideal $\D_k{}^{\,}\g^{\psi+k\cdot\phi}\subset\D_k$. Therefore, the
assignment $1\mto e_k$ gives  a well defined map
$$
q_k: \
\D_k/\D_k{}^{\,}\g^{\psi+k\cdot\phi}\too\jmath_*\CE_k .
$$

But, $e_k$ is a generator of the $\D_k$-module $\jmath_*\CE_k$.
Therefore, the map $q_k$ is surjective, and we are done.
\end{proof}

For any $k\in\Z,$ the TDO $\D_k$ comes equipped with an increasing filtration
such that $\gr\D_k\cong p_*\CO_{T^*X},$ where
$p: T^*X\to X$ is the bundle projection.
Let  $[SS(\CM)]$ denote
the characteristic {\em cycle} of a holonomic
$\D_k$-module, resp. $[\ssup(\CF)]$ denote the support cycle
of a coherent sheaf on ~$T^*X$.

\begin{lemma}
\label{SS} 
Conditions \eqref{item1}-\eqref{item3} imply that 
$\D_k/\D_k{}^{\,}\g^{\psi+k\cdot\phi}$ is a  holonomic
$\D_k$-module and one has
an equality of Lagrangian cycles
$$[SS(\D_k/\D_k{}^{\,}\g^{\psi+k\cdot\phi})]=[SS(\jmath_*\CE_k)],\qquad\forall k\in\Z.$$
\end{lemma}

\begin{proof} Given a pair of Lagrangian algebraic cycles
$Y,Y'\sset T^*X$, we write $Y\leq Y'$ whenever
the cycle $Y'-Y$ is a nonnegative integer combination
of irreducible Lagrangian subvarieties.

 The filtration on $\D_k$ induces a natural filtration
on the $\D_k$-module $\D_k/\D_k{}^{\,}\g^{\psi+k\cdot\phi}$.
One has a canonical {\em surjection} of  graded $\CO_X$-modules:
 $\gr\D_k/(\gr\D_k)\g\onto \gr(\D_k/\D_k{}^{\,}\g^{\psi+k\cdot\phi}).$
Furthermore,  there is a natural isomorphism
$\gr\D_k/(\gr\D_k)\g\cong p_\idot(\oo_{T^*X}|_{\mu\inv(0)})$.
Thus, we deduce 
$$[SS(\D_k/\D_k{}^{\,}\g^{\psi+k\cdot\phi})]\leq
[\ssup(\oo_{T^*X}|_{\mu\inv(0)})]=[\mu^{-1}(0)].
$$

Hence, Lemmas \ref{cond_lem} and \ref{muflat} imply
 that $\D_k/\D_k{}^{\,}\g^{\psi+k\cdot\phi}$ is a  holonomic
$\D_k$-module and one has
\begin{equation}\label{ineq}
[SS(\D_k/\D_k{}^{\,}\g^{\psi+k\cdot\phi})]\leq [\mu\inv(0)]=[\bmu\inv(0)\cap\bnu\inv(0)]=
 \underset{a\to0}\lim\, [\bmu\inv(a\cdot\phi)\cap\bnu\inv(a)].
\end{equation}

Next recall that, according to \cite{Gi}, Theorem~6.3, one has 
 an equality of Lagrangian cycles 
$[SS(\jmath_*\CE_k)]=\underset{a\to0}\lim\, 
[\bmu\inv(a\cdot\phi)\cap\bnu\inv(a)]$. 
Thus, from \eqref{ineq} we deduce 
$$[SS(\D_k/\D_k{}^{\,}\g^{\psi+k\cdot\phi})]\leq[SS(\jmath_*\CE_k)].$$

On the
other hand,  Lemma~\ref{zabyl} yields an opposite inequality, and we are
done.
\end{proof} 

To complete the proof of Proposition \ref{thm_cycle}, we observe
that the
characteristic cycle of the kernel of the
surjective $\D_k$-module
map
$\D_k/\D_k{}^{\,}\g^{\psi+k\cdot\phi}
\twoheadrightarrow\jmath_*\CE_k$ (for $k\ll0$) is
equal to zero,
due to Lemma \ref{SS}.
Hence the kernel itself is zero, and we are done.
\qed

\section{A triple flag variety.} 
\subsection{Horocycle spaces}
Let $G$ be a connected complex semisimpe group. 
Let $\BT$ be  the {\em abstract} Cartan torus of $G$, and
$W$  the corresponding abstract Weyl group of $G$.
Fix a Borel subgroup $B\sset G$, with the unipotent
radical $N$. Thus, we have $\BT=B/N$.

Let $\CB=G/B$ be the flag variety of $G$.
There is a canonical $G$-equivariant $\BT$-torsor
$\tB\to\CB,$ where $\tB=G/N$ is the {\em base affine
space}, and where $\BT$ acts on $G/N$ on the right.
The $G$- and $\BT$-actions on  $\tB$ commute, hence, make $\tB$
 a $G\times\BT$-variety.

Given elements $x,y\in W$, let $\BT_{x,y}$ be the image of
the torus imbedding
$\BT\into \BT\times \BT,\ t\mto x(t)\times
y(t)$.
The {\em horocycle space} associated with the pair $(x,y)$ is defined to be
 $\tB_{x,y}:=(\tB\times\tB)/\BT_{x,y}.$
There are two natural projections $\pr_x,\ \pr_y:\ \tB_{x,y}\onto\CB$,
on the first and second factor, respectively.
The left $G\times G$-action  and  the right 
$\BT\times\BT$-action make 
 $\tB_{x,y}$ a smooth $G\times G\times\BT\times
\BT$-variety. 
The right $\BT\times\BT$-action clearly
 factors through 
 the torus  $(\BT\times
\BT)/\BT_{x,y}$.
This makes the map
$\pr_x\times\pr_y:\ \tB_{x,y}\onto \CB\times\CB$ a $G\times G$-equivariant
$(\BT\times
\BT)/\BT_{x,y}$-torsor.

\subsection{}\label{hor} Write $\g,\bt$ for the Lie algebras of
the groups $G$
and $\BT,$
respectively. 
Let $\Ug$ and $\Ut$ be the corresponding enveloping algebras.
Let $\FZ$ denote the center of $\Ug$. We may (and will)
view $\Ut$ as a $\FZ$-module via  the Harish-Chandra
homomorphism $\hcc:\ \FZ\to\Ut$.

The differential of the $G\times\BT$-action on $\tB$
induces an algebra
map
$\kappa:\ \Ug\o\Ut\to\Gamma(\tB, \D_\tB)$.
It is clear that the image of $\kappa$ is contained
in $\Gamma(\tB, \D_\tB)^\BT$, the subalgebra 
of right $\BT$-invariant differential operators.
It is also easy to see that, for any  $z\in \FZ,$ one has
$\kappa(z\otimes1)=\kappa(1\otimes\hcc(z)).$
 Furthermore, it was shown in \cite{BoBr} that
 the  map $\kappa$ gives rise to an
algebra isomorphism
\beq{bobr}
\kappa:\
\Ug\o_\FZ\Ut\ \iso \ \Gamma(\tB, \D_\tB)^\BT.
\eeq

Let $\varpi_y$ be the pull-back of  $\om_\CB$, the
canonical bundle on the flag manifold $\CB$, via the projection $\pr_y:
\tB_{x,y}\onto \CB$, and define

\beq{ddy}\dd{x}{y}:=\varpi_y\o_{\oo_{\tB_{x,y}}}\D_{\tB_{x,y}}\o_{\oo_{\tB_{x,y}}}
\o\varpi_y^{-1}.
\eeq 

It is clear that $\dd{x}{y}$ is 
a TDO on $\tB_{x,y},$ moreover,
the pair $(\dd{x}{y},\ G\times G\times\BT\times\BT)$ is
 a
Harish-Chandra algebra.
Thus, one has
a canonical  algebra
map
$\Ug\o\Ug\o\Ut\o\Ut\to\Gamma(\tB_{x,y},\ \dd{x}{y})$.
We let ${\mathfrak T}_{x,y}$ denote an ideal of the algebra $\Ut\o\Ut=\UU(\bt\oplus\bt)$
generated by $\bt_{x,y}:=\Lie \BT_{x,y}$, 
a vector subspace of $\bt\oplus\bt\sset \UU(\bt\oplus\bt).$

From the  Borho-Brylinski isomorphism \eqref{bobr}
one derives an algebra isomorphism
\beq{bobr2}\kappa:\ (\Ug\o\Ug)\ \bo_{\FZ\o\FZ}\
\big[(\Ut\o\Ut)/{\mathfrak T}_{x,y}\big]\
\iso\
\Gamma(\tB_{x,y},\
\dd{x}{y})^{\BT\times\BT}.
\eeq

A $\D_{x,y}$-module
 is said to be {\em monodromic}
provided it is a  holonomic
$\D$-module which is, moreover,
$G$-equivariant with respect to
the  $G$-diagonal left action on $\tB_{x,y}$ and is also
monodromic with respect to
the right  $\BT\times\BT$-action on $\tB_{x,y}$.
\vskip 2pt

\npb{Let $\Mon(\tB_{x,y})$  denote the full abelian subcategory
of $\Lmod{\D_{x,y}}$ whose objects
are monodromic modules.}
\vskip 2pt

 A {\em complex} $\CM^\hdot\in D^b(\Lmod{\D_{x,y}})$
 is said to be  monodromic provided,
for  each integer $\ell\in \Z$, the cohomology sheaf
${\mathcal H}^\ell(\CM^\hdot)$ is a  monodromic $\D_{x,y}$-module.
\vskip 2pt
Write $\Lmod{A}$, resp. $D^b(\Lmod{A})$, for the abelian category,
resp. bounded derived category, of left modules
over $A$, an associative algebra or a sheaf of associative algebras.
\vskip 2pt

\npb{Let $\Don(\tB_{x,y})$  denote the full triangulated subcategory
of $D^b(\Lmod{\D_{x,y}})$ whose objects
are monodromic complexes.}

\subsection{}\label{sec_lamunu} {\bf In the rest of the paper we let
$G=SL_n.$}
We put  $\g:=\Lie G=\sv_n.$

The group $G$ acts naturally on the projective space $\P=\P(\C^n)$,
 and we let
$G$ act diagonally on  $\CB\times\CB\times\P$,  a `triple flag manifold'.
There is a  moment map 
 associated with the induced Hamiltonian  $G$-action on
$T^*(\CB\times\CB\times\P),$
 the total space of the
cotangent bundle.

The crucial geometric properties of 
the triple flag manifold are summarized in the following 
proposition.

\begin{proposition}\label{ein moment} \vi The group
 $G$ acts diagonally on $\CB\times\CB\times\P$ with finitely many orbits.
The orbits are parametrized by the set
of pairs $(w,\sigma),$ where $w\in \syn$ is a permutation 
 and $\sigma$ is a decreasing subsequence of 
the sequence of integers $w(1),\ldots,w(n)$.
\vskip 2pt

\vii The  moment map 
$\dis\,\mu:\
T^*(\CB\times\CB\times\P) \to \g^*\,$ is {\em flat}.
\end{proposition}

\proof Part (i) is proved in \cite{MWZ}, 2.11.
Part (ii) follows from  Lemma \ref{muflat}(i)-(ii) and the equation
$$\dim(\CB\times\CB\times\P)=
\frac{n(n-1)}{2}+\frac{n(n-1)}{2}+(n-1)=n^2-1=\dim G.\eqno\Box$$
\smallskip

Next, we fix a pair of elements $x,y\in W$ and put $\yy_{x,y}=\tB_{x,y}\times\P$.
Below, we will be interested in the left $G$-diagonal action on
 $\yy_{x,y}$ as well as in the right $\BT\times\BT$-action
induced by the one on  $\tB_{x,y}$, the first factor, and  trivial along 
$\P$, the second factor. 
 The right $\BT\times\BT$-action clearly factors through
a $(\BT\times\BT)/\BT_{x,y}$-action. It follows that
 the  left  $G$-action  
factors through the adjoint group $G/\text{Center}(G).$
These left and right actions make $\yy_{x,y}$ a
$G\times\BT\times\BT$-variety, and there is an
 induced Hamiltonian  $G\times\BT\times\BT$-action on
the cotangent bundle
$T^*\yy_{x,y}$. 

Let $\bt_{x,y}^\perp\sset \bt^*\oplus\bt^*$ denote the
annihilator of the subspace $\bt_{x,y}\sset \bt\oplus\bt.$
It is immediate to see that the image of an associated moment map
is automatically contained in the subspace $\g^*\times\bt_{x,y}^\perp\sset \g^*\times
\bt^*\times\bt^*.$

From Proposition \ref{ein moment} one easily deduces the following

\begin{corollary}\label{zwei}
\vi The group
 $G\times\BT\times\BT$ acts on $\yy_{x,y}$ with finitely many orbits.

\vii The  moment map 
$\dis\,
\mu:\
T^*\yy_{x,y}\too\g^*\times\bt_{x,y}^\perp
\en$ is {\em flat}.\qed
\end{corollary}

\subsection{}\label{bbmon}
For any $c\in\C$, let $\D_{\P}^{c}$ denote the sheaf of TDO on
$\P$ with twist `$c$'.
Given $x,y\in W,$ we write
$\D_{x,y}^c:=\D_{x,y}\boxtimes
\D_{\P}^{c}$ for the corresponding TDO on $\yy_{x,y}=\tB_{x,y}\times\P$.
Thus, $(\D_{x,y}^c,\ 
G\times\BT\times\BT)$ is a Harish-Chandra algebra.

We have the notion
of a monodromic $\D^c_{x,y}$-module, resp. monodromic complex,
defined the same way as we have done in \S\ref{hor}, with
the space $\tB_{x,y}$ being now replaced by $\yy_{x,y}$.
Similarly to section \ref{hor}, one 
defines an abelian category
 $\Mon_c(\yy_{x,y})
\sset \Lmod{\D^c_{x,y}},$ of monodromic $\D^c_{x,y}$-modules,
 resp. triangulated category
$\Don_c(\yy_{x,y})
\sset D^b(\Lmod{\D^c_{x,y}}),$ of monodromic complexes.

Corollary \ref{zwei}(i) implies the following
\begin{corollary}\label{hholl}
Any object of $\Mon_c(\yy_{x,y})$ is a holonomic $\D$-module
with regular singularities.
\end{corollary}

Let $\bt^*_\Z\sset\bt^*$ denote the weight lattice of $\BT$
and, given $x,y\in W$, write $\pr_{x,y}=\pr_x\times\pr_y\times \Id_\P:\
\yy_{x,y}=\tB_{x,y}\times\P\to \CB\times\CB\times\P$ for the natural projection.

For any monodromic complex $\CV\in\Don_c(\yy_{x,y})$
there is a natural monodromy action of
the  fundamental group of the torus $\BT\times\BT$
 in the stalks of 
$(\pr_{x,y})_\idot{\mathcal H}^\hdot(\CV)$,
a sheaf theoretic  direct image of the cohomology of $\CV$.
Note that one may view the
fundamental group of $\BT\times\BT$ as a lattice $\pi_1(\BT\times\BT)\sset\bt\oplus\bt$.

Let $\bar\la,\bar\nu\in\bt^*/\bt^*_\Z$. We say that  $\CV\in\Don_c(\yy_{x,y})$ has
monodromy $(\bar\la,\bar\nu)$ provided, for any element 
$(t_1,t_2)\in\pi_1(\BT\times\BT)\sset \bt\times\bt$, 
 the corresponding  monodromy  $(t_1,t_2)$-action in 
$(\pr_{x,y})_\idot{\mathcal H}^\hdot(\CV)$ is a linear operator of the form $\dis
e^{2\pi\sqrt{-1}(\langle\la,t_1\rangle+\langle\nu,t_2\rangle)}\Id + u,$ where 
$u$ is a nilpotent operator. In this formula, $\langle\la,t_1\rangle,$
resp $\langle\nu,t_2\rangle$, denotes the value at  $t_i\in\bt$
of an arbitrary representative in $\bt^*$ of the corresponding element
 $\bar\la,\bar\nu\in\bt^*/\bt^*_\Z$. 
\vskip 2pt

\npb{Let $\Mon_{\bar\la,\bar\nu,c}(\yy_{x,y})$ be the full abelian
subcategory
of $\Mon_c(\yy_{x,y})$, resp.
 $\Don_{\bar\la,\bar\nu,c}(\yy_{x,y})$  be the full triangulated subcategory
of  $\Don_c(\yy_{x,y})$, whose objects
have monodromy ~$(\bar\la,\bar\nu)$.}
\vskip 2pt

It is clear that the category  $\Don_{\bar\la,\bar\nu,c}(\yy_{x,y})$ is trivial
unless the character $(\bar\la,\bar\nu)$ restricts to the unit
character of the subgroup $\pi_1(\BT_{x,y})\sset \pi_1(\BT\times\BT)$.
Let $\bar\bt_{x,y}^\perp\sset \bt^*/\bt^*_\Z\oplus\bt^*/\bt^*_\Z$
be the set of such characters, i.e., the set pairs $(\bar\la,\bar\nu)$ such
that the linear function $(\la,\nu)$ takes integer values
on the lattice $\pi_1(\BT_{x,y})$, for some
(equivalently, any)  representative $(\la,\nu)\in\bt^*\oplus\bt^*$
of the element $(\bar\la,\bar\nu)$. 

Thus, there is
 a canonical direct sum decomposition
$$
\Mon_c(\yy_{x,y})=\underset{(\bar\la,\bar\nu)\in\bar\bt_{x,y}^\perp}{\mbox{$\bigoplus$}}
\Mon_{\bar\la,\bar\nu,c}(\yy_{x,y}),
\quad\text{resp.}\quad
\Don_c(\yy_{x,y})=\underset{(\bar\la,\bar\nu)
\in\bar\bt_{x,y}^\perp}{\mbox{$\bigoplus$}}
\Don_{\bar\la,\bar\nu,c}(\yy_{x,y}).
$$
 
For any object $\CV\in \Don_c(\yy_{x,y})$, we write
$\CV=\oplus_{\bar\la,\bar\nu}\,\CV^{(\bar\la,\bar\nu)}$ for the 
corresponding direct sum decomposition. 

\subsection{Convolution}\label{conv} Let $\rho\in\bt^*$
denote the half-sum of positive roots.
Associated with any triple $x,y,z\in W$, 
there is a standard convolution functor, cf. \cite{BG}, \S 5:
\beq{convolution}
*:\
\Don_{\bar\la,\bar\nu,c}(\yy_{x,y})\times
\Don_{-\bar\nu, \bar\la',c'}(\yy_{y,z})\too
\Don_{\bar\la,\bar\la',c+c'}(\yy_{x,z}).
\eeq

To define convolution \eqref{convolution}, one writes $\BT_{x,y,z}$ for
the image of
a torus imbedding
$\BT\into \BT\times \BT\times \BT,$ given by
$t\mto x(t)\times y(t)\times z(t).$ Clearly,
 $(\tB\times\tB\times\tB)/\BT_{x,y,z}$ is a smooth
variety and there are 3 natural projections  along various factors
$\tB$,
eg.
 ${\mathsf q}_{x,y}:\ (\tB\times\tB\times\tB)/\BT_{x,y,z}\to
\tB_{x,y}$.
We may extend each of these  morphisms
to a cartisian product with a copy of the
projective space $\P$. This way, we get e.g.
a map ${\mathsf q}_{x,y}\boxtimes\Id_\P:\
[(\tB\times\tB\times\tB)/\BT_{x,y,z}]\times\P\to
\tB_{x,y}\times\P$, and also a  map ${\mathsf q}_y\boxtimes\Id_\P$.

There is  also a natural map
${\mathsf q}_y: 
(\tB\times\tB\times\tB)/\BT_{x,y,z}\onto\tB$, the projection
to the middle factor.
The definition of the TDO $\dd{x}{y}$, see \eqref{ddy},
combined with the canonical isomorphism
$\D_\CB^\opp\cong \om_\CB\o_{\CO_\CB}\D_\CB\o_{\CO_\CB}\om_\CB\inv$
insure that, for any $\dd{x}{y}$-module $M$,
the sheaf ${\mathsf q}_{x,y}^*M$ has a natural right
${\mathsf q}_y^*\D_\tB$-module structure.

In formula \eqref{star} below, we  use simplified notation and write
 ${\mathsf q}_{x,y}$ for ${\mathsf q}_{x,y}\boxtimes\Id_\P$.
With that understood,
the  convolution of any pair of objects 
$\CM\in \Don_{\bar\la,\bar\nu,c}(\yy_{x,y})$ and
$\CN\in \Don_{-\bar\nu,\bar\la',c'}(\yy_{y,z})$ is defined by
\beq{star}\CM * \CN:=(R{\mathsf q}_{x,z})_\bullet\big({\mathsf q}_{x,y}^*\CM\
\stackrel{L}{\bo}_{({\mathsf q}_y^*\D_\tB)\boxtimes\oo_\P}\
{\mathsf q}_{y,z}^*\CN\big).
\eeq

\subsection{}\label{R} Let $w_0\in W$ denote the  element of
maximal
length. The  variety $\tB_{1,w_0}$ will play a special
role.
This variety has a $G$-equivariantly trivialized
 canonical bundle. Note also that, for any $y\in W$, we have
$\tB_{y,w_0y}=\tB_{1,w_0}$, hence $\yy_{y,w_0y}=\yy_{1,w_0}.$

We are going
to introduce, following  \cite{BG}, \S5,  a pair of  pro-objects of the
category $\Lmod{\D_{y,w_0y}}$ that will be important  in subsequent sections.
To this end,  view $\tB_{y,w_0y}$
as a $G\times \BT\times\BT$-variety
where $G$, the first factor,
acts diagonally on  $\tB_{y,w_0y}$ on the left,
and the group $\BT\times\BT$ acts naturally on the right.
Let $ {\scr U}$ be the unique Zariski open dense
$G\times \BT\times\BT$-orbit in $\tB_{y,w_0y}$.
We have the following  diagram involving  natural projections 
and an open imbedding:
\beq{jf}
\xymatrix{
G\backslash{{\scr U}}\ 
&\ {\scr U}\ \ar@{->>}[l]_<>(0.5){h}
\ar@{^{(}->}[r]^<>(0.5){\jmath}&
\ \tB_{y,w_0y}&\yy_{y,w_0y}=\tB_{y,w_0y}\times\P\ar@{->>}[l]_<>(0.5){\pr_1}.
}
\eeq

For any $u\in {\scr U}$, the isotropy group of $u$
with respect to the left $G$-diagonal action is 
a maximal torus in $G$, in particular, it is
connected.
We have an isomorphism $(\BT\times\BT)/\BT_{y,w_0y}\iso\BT,\
(t_1,t_2)\mto t_1.$
The right $\BT\times\BT$-action
on $\scr U$ descends to a well defined
$(\BT\times\BT)/\BT_{y,w_0y}$-action
on the orbit space $G\backslash{ {\scr U}}$
and makes the latter space a $\BT$-torsor.

Define a linear involution
\beq{dag}
(-)^\dag:\ \bt^*\to\bt^*,\en \la\mto\lambda^\dag:=-w_0\lambda.
\eeq

Given  $\la\in \bt^*$, let ${\mathcal J}^\la$ be a rank 1 local
system on the $\BT$-torsor $G\backslash{ {\scr U}}$
 with monodromy $e^{2\pi\sqrt{-1}\la}$.
Thus, $h^*({\mathcal J}^\la)$, cf. \eqref{jf}, is a  local
system on ${\scr U}$ with monodromy $(\la,\la^\dag)\in \bt^*\times\bt^*.$
Further, let $\CE^\infty$ be a  pro-unipotent local system on  $G\backslash{ {\scr U}}$,
 cf. \cite{BG}, page 18. Such a local system is unique
up to isomorphism, and we put $\wh{\mathcal J}^\la:={\mathcal J}^\la\o\CE^\infty.$
All stabilizers of  the
$G$-action in  $ {\scr U}$ being connected,
it follows that $h^*(\wh{\mathcal J}^\la)$ is a
(projective limit of) $G$-equivariant
local systems on $ {\scr U}$.

We use diagram \eqref{jf} to define the following pair of
(projective limits of)
$G$-equivariant holonomic $\D_{y,w_0y}$-modules (cf. also 
\cite[formulas (5.9.2), (5.22)]{BG}):
\beq{rr}
\CR^{\la,\la^\dag}_*:=\pr_1^*\ccirc \jmath_*\ccirc h^*(\wh{\mathcal J}^\la),\en\;\CR^{\la,\la^\dag}_!:=\pr_1^*\ccirc \jmath_!\ccirc h^*(\wh{\mathcal
J}^\la)
\in \textsl{lim}^{\,} \textsl{proj}\,\Mon_{\bar\la,\bar\la^\dag,0}(\yy_{y,w_0y}).
\eeq

For any $x,y\in W$ and $\nu\in\bt^*$,
 convolution with the above objects yields the following 
{\em mutually inverse} (triangulated) equivalences, see \cite{BG} \S 5:
\begin{align}
&\Don_{\bar\nu,-\bar\la,c}(\yy_{x,y})\too
\Don_{\bar\nu,\bar\la^\dag,c}(\yy_{x,w_0y}),\quad
\CM\mapsto \CM*\CR^{\la,\la^\dag}_!\label{rconv}\\
&\Don_{\bar\nu,\bar\la^\dag,c}(\yy_{x,w_0y})\too\Don_{\bar\nu,-\bar\la,c}(\yy_{x,y}),
\quad
\CV\mapsto \CV*\CR^{-\la^\dag,-\la}_*.\nonumber
\end{align}


\section{Mirabolic $\D$-modules}
\subsection{Character $\D$-modules}\label{hor2} Recall
that  $G=SL_n$ and write  $\D_G$ for the sheaf of differential operators on
$G$.

The group $G\times G$ acts on $G$  by
left and right translations,  via $\ (g_1\times g_2):\ g\mto g_1gg_2\inv$.
The differential of that action induces a linear
map from $\g\times\g$ to the Lie algebra of vector fields on $G$.
The latter map extends uniquely to  an associative algebra
homomomorphism $\lr:\ \Ug\otimes(\Ug)^\opp\to\D(G)$.
 
It is clear that the sets
 $\lr(\FZ\o 1),\  \lr(1\o \FZ)$ are contained
in $\D(G)^{G\times G}\sset \D(G)$, the subalgebra of
 {\em bi-invariant} differential operators on $G$. 
In fact, it is easy to see that one has
$\lr(z\otimes1)=\lr(1\otimes \tau(z))$,
where $\tau:\ \Ug\iso(\Ug)^\opp$ is an algebra isomorphism
defined on generators $x\in\g$ by
$\tau(x)=-x$. 
In this way,
one obtains a well defined, injective
 algebra morphism 
$$\lr\ccirc (\Id\times\tau):\
\Ug\otimes_\FZ\Ug\ \into\ \D(G),\quad
u\o u'\mto
\lr(u\o\tau(u')).$$

Next, we put   $\X:=G\times\P$. Recall
that $\D_{\P}^{c}$ stands for the sheaf of TDO on
$\P=\P^{n-1},$ with twist  $c\in\C$. We
let $\dis\dxc:=\D_G\boxtimes\D_{\P}^{c}$
be an associated TDO on $\X$, and write $\D(\P,c):=\Ga(\P,
\D_{\P}^{c}),$ resp.
$\D(\X,c)=\Ga(\X, \dxc)=\Ga(G,\D_G)\o \D(\P,c),$
for the corresponding algebras of global sections.
Let $\zeta$ denote the composite of the following algebra maps
\beq{lrc}
\zeta:\
\xymatrix{
 (\Ug\otimes_\FZ\Ug) \o \D(\P,c)\
 \ar@{^{(}->}[rrr]^<>(0.5){[\lr\ccirc(\Id\o\tau)]\o\Id}&&&
\ \D(G) \o \D(\P,c) \ar@{=}[r]& \D(\X, c).
}
\eeq

Further, we have an obvious imbedding $\upsilon: \FZ \into (\Ug\otimes_\FZ\Ug) \o \D(\P,c),\
z\mto (z\o1)\o 1=(1\o z)\o 1$.
\medskip

Recall  that an action of an algebra $A$  on an $A$-module
$M$ is said to be {\em locally finite} if, for
any $m\in M$, one has $\dim A\cd m<\infty$.

Now, we let $G$ act on itself via the adjoint action, and let $G$ act
diagonally on  $\X=G\times\P$.

\begin{definition} A $\dxc$-module $\CF$ 
 is called  a (mirabolic) {\em character
module} if the following two conditions hold,
cf. \cite{GG}, Remark 5.2, and also \cite{FG}, \S\S 4-5:

\begin{enumerate}

\item $\CF$ is a $G$-equivariant $\D$-module with respect to the $G$-diagonal
action on $\X=G\times\P$.

\item
The left  $\FZ$-action on $\Gamma(\X,\CF)$, via the  imbedding $\zeta\ccirc\upsilon$, is locally finite.
\end{enumerate}
\end{definition}

\npb{Let  $\ccat$ be the full abelian
subcategory
of $\Lmod{\D_{\X}^{c}}$ whose objects are  character
modules.}

\subsection{Spectral decomposition} \label{Xsec}
Given a commutative algebra $A$, write $\max(A)$ for the set of maximal
ideals in $A$. For  $\fa\in\max(A)$,
and an $A$-module $M$, we put
\beq{fa}
M^{(\fa)}:=\{m\in M\mid \exists \ell=\ell(m)\gg0\en
\text{such that}\en \fa^\ell\cdot m=0\}.
\eeq

If, moreover, the $A$-action on $M$ is locally finite, then one
has a canonical $A$-stable {\em spectral decomposition}
\beq{sp}M=\mbox{$\bigoplus$}_{\fa\in\max(A)}\ M^{(\fa)}.\eeq

First, we apply \eqref{sp} in the case where
$A:=\Ut\o\Ut$.  Given $\la\in\bt^*$, let 
$I_\la\sset\Ut$ denote a maximal
ideal generated by the elements $\{x-\la(x),\ x\in\bt\}$.
Similarly,  associated with any $\la\times\nu\in\bt^*\times\bt^*,$
there is a maximal ideal
$I_{\la,\nu}:=I_\la\o\Ut+\Ut\o I_\nu\in\max(\Ut\o\Ut).$
Thus,  for a locally finite $\Ut\o\Ut$-module $M$,
one has a spectral decomposition
$M=\bigoplus_{\la\times\nu\in\bt^*\times\bt^*}\
M^{I_{\la,\nu}}.$

Next, we take
 $A=\FZ$, the center of $\Ug$. Given $\theta\in\bt^*/W$, let
$\FZ_\theta\sset\FZ$ be the maximal ideal corresponding to
$\theta$ via the Harish-Chandra isomorphism $\hcc: \FZ\iso\C[\bt^*/W].$
Thus, for any locally finite $\FZ$-module $M$, one has a spectral
decomposition $\dis M=\bop_{\th \in\bt^*/W}\ M^{(\FZ_\theta)}$.

Similarly,
let $M\in D^b(\Lmod{\FZ})$ be a complex such
the induced $\FZ$-action on $H^\hdot(M)$ is locally finite.
Then, one can show that
there is a canonical  direct sum
decomposition 
$M=\bop_{\th \in\bt^*/W}\ M^{(\th)}$ where
$M^{(\th)}\in D^b(\Lmod{\FZ})$ are such that, for
each $\th$,
one has $H^\hdot(M^{(\th)})=[H^\hdot(M^{(\th)})]^{(\FZ_\theta)}$.

We write
$\FZ_\th^c\sset \D(\X, c)$ for $\zeta\ccirc\upsilon(\FZ_\th)$, the image of
$\FZ_\th$ under the composite imbedding $\zeta\ccirc\upsilon: \FZ\to\D(\X,
c)$, see \eqref{lrc}.
For any character $\D$-module
$\CF\in {\scr C}_c(\X)$, one 
has a vector space decomposition
$H^\hdot(\X,\CF)=\bop_{\th\in \bt^*/W}\ 
H^\hdot(\X,\CF)^{(\FZ_\th^c)}.$

It will be convenient to introduce 
a semi-direct product $\waff:=W\ltimes\bt^*_\Z$,
an (extended) affine Weyl group. The group $\waff$ acts naturally on $\bt^*$ by
affine linear transformations.  Let  $\Th\in \bt^*/\waff$.
It is easy to see that,
for  any
$\CF\in {\scr C}_c(\X)$, the vector space
$$H^\hdot(\X,\CF)\langle\Th\rangle
:=\bop_{\{\th\in \bt^*/W\,\, |\,\, \th\,(\operatorname{mod}\, \waff)=\Th\}}\ 
H^\hdot(\X,\CF)^{(\FZ_\th^c)}$$
is a $\D(\X, c)$-submodule in 
$H^\hdot(\X,\CF)$. Furthermore, there is
a unique $\D$-submodule $\CF\langle\Th\rangle\sset \CF$
such that, one has $H^\hdot(\X,\CF\langle\Th\rangle)=H^\hdot(\X,\CF)\langle\Th\rangle$.

Let ${\scr C}_{\Th,c}(\X)$ be a full subcategory
of ${\scr C}_c(\X)$ whose objects satisfy
$\CF=\CF\langle\Th\rangle.$
This way,
following the strategy of \cite{G}, one obtains
 a canonical spectral decomposition,
cf. \cite{G}, Theorem 1.3.2:
\beq{ccharacter}
{\scr C}_c(\X)=\underset{\Th \in\bt^*/\waff}{\bop}\ {\scr C}_{\Th,c}(\X),
\quad \CF=\underset{\Th \in\bt^*/\waff}{\bop}\ \CF\langle\Th\rangle,
\qquad \CF\langle\Th\rangle\in {\scr C}_{\Th,c}(\X),\ \forall \Th.
\eeq

\npb{Let  $\dcat$, resp.  $\dcatchi$, denote the full triangulated subcategory
of $D^b(\Lmod{\D_{\X}^{c}})$ whose objects
are complexes $\CF^\hdot$ such that for the corresponding
 cohomology sheaves, one has
${\mathcal H}^j(\CF^\hdot)\in \ccat,$
resp. ${\mathcal H}^j(\CF^\hdot)\in \ccat$ and  ${\mathcal H}^j(\CF^\hdot)=[{\mathcal
H}^j(\CF^\hdot)]\langle\Th\rangle,$ for any 
$j\in \Z$, resp.  any $\Th \in\bt^*/\waff$.}
\bigskip

\subsection{The functors $\ch$ and $\hc$}
\label{dva} An important role in what follows will be played by a diagram
\beq{maps}
\xymatrix{
&\ G\times(\tB/\BT)\times\P\ \ar[dl]_<>(0.6){p}\ar[dr]^<>(0.6){q}&\\
{\BX=G\times\P} && \yy_{1,1}
}\eeq
In this diagram, 
the maps $p$ and $q$
are given by $p(g,\tilde{x},l):=(g,l)$ and
$q(g,\tilde{x},l):=(g\tilde{x},\tilde{x}, g l)$, respectively.
The group $G$ acts  diagonally on all cartesian products involved
in the above diagram, and acts on $G$ by
conjugation.  With these $G$-actions, all maps in
\eqref{maps} become $G$-equivariant morphisms.

Recall the notation $\om_y$ for the
pull-back of  $\omega_\CB$, the canonical bundle on the flag manifold,
via the second projection $\tB_{x,y}=(\tB\times\tB)/\BT_{x,y}\to\CB$.
Given a quasi-coherent sheaf $\CV$ on $\yy_{x,y}=\tB_{x,y}\times\P$, we will often
abuse the notation and write $\om^{\pm 1}\o\CV$ for
$\dis(\om^{\pm 1}_y\boxtimes\CO_\P)\o_{_{\CO_{\tB_{x,y}\times\P}}}\CV$.


Now, fix  $\la\in\bt^*$, and let $\bar\la\in\bt^*/\bt_\Z^*$, resp.
$\Th\in \bt^*/\waff$, be the image of $\la$.
We use diagram \eqref{maps} and introduce a pair of  functors,
$\ch$ and $\hc$,
given, for any $\CF\in\dcatchi,$
 resp. $\CV\in \Don_{\bar\la,\bar\la^\dag,c}(\yy_{1,w_0}),$ by the
formulas, cf. \eqref{rconv} and \cite[\S 8]{G},
\begin{align*}
\HC(\CF):=
\big(\om\inv\o q_!p^*(\CF)\big)
\,*\,\CR^{\la,\la^\dag}_![\dim\CB],\qquad
\ch(\CV):=p_*q^!\big(\om\o (\CV*\CR^{-\la^\dag,-\la}_*)\big)[-\dim\CB].
\end{align*}

\begin{proposition}\label{hc_prop} \vi The above formulas give
an adjoint pair, $(\ch,\hc)$, of triangulated functors
$$\xymatrix{
\dcatchi{\hphantom{x}}
\ar@<0.5ex>[rr]^<>(0.5){\,\HC\,}&&
{\hphantom{x}}
\Don_{\bar\la,\bar\la^\dag,c}(\yy_{1,w_0}).\ar@<0.5ex>[ll]^<>(0.5){\,\ch\,}
}
$$

\vii   Any character module $\CF\in{\scr C}_{\Th,c}$ is 
regular holonomic and the functor
 $\ch\ccirc\HC$ contains  $\Id_{\dcat}$, the identity functor,
as a direct summand.
\end{proposition}

The proof of the proposition repeats word by word the proof of Theorem~4.4
  of~\cite{MV} or ~\cite{G},~\S9, using
Corollary \ref{hholl}.

We can now formulate one of the main results of the paper.

\begin{theorem}\label{main}   For any 
regular dominant weight $\la\in\bt^*$ and  $c\neq -1,-2,\ldots,1-n$,
the functor $\HC$ restricts to an
exact functor $\,\HC:\ {\scr C}_{\Theta,c}(\X)\to
\Mon_{\bar\la,\bar\la^\dag,c}(\yy_{1,w_0})$, between 
{\em abelian}
categories.
\end{theorem}

This theorem is a `mirabolic analogue' of a result
of Bezrukavnikov, Finkelberg, and Ostrik, ~\cite{BFO}.

Theorem \ref{main} will be deduced from a more precise result,
 Theorem \ref{hk} of the next subsection.

\subsection{}\label{Ysec} 
We will use simplified notation
 $\D(\P,c):=
\Ga(\P,\D^c_\P),$ resp. $\D(\X,c):=\Ga(\X,\ \D_\X^c),$ and
 $\D(\yy_{x,y},c):=\Ga(\yy_{x,y},\ \D^c_{x,y})$, for any $x,y\in W$.

An important role  below will be played by
 the algebra
\beq{buc}\BU_c:=(\Ug\o\Ug)\, \bo_{\FZ\o\FZ}\, (\Ut\o\Ut)\, \o\,
 \D(\P,c).
\eeq
The algebra \eqref{buc} fits into a
 diagram of  algebra maps
\beq{ucc}
\xymatrix{
&&\,\BU_c\,\ar@{->>}[dl]_<>(0.7){\eta=pr\o a\o \Id_{_{\,}}\;}
\ar@{->>}[dr]^<>(0.7){\ \kappa_{x,y}\o\Id}&&\,\Ut\o\Ut\,\ar[ll]_<>(0.5){\eps}\\
\D(\X,c)\ &\ (\Ug\o_\FZ\Ug)\o \D(\P, c)\ar@{_{(}->}[l]_<>(0.5){{}^\zeta}&&
\D(\yy_{x,y},c)^{\BT\times\BT}\ \ar@{^{(}->}[r]^<>(0.5){{}^j}&\ \D(\yy_{x,y},c).
}
\eeq

In this diagram, 
the map $j$ is the natural inclusion, the map $\zeta$
is the isomorphism \eqref{lrc}, and  the map $\kappa_{x,y}$  comes from 
the isomorphism $\kappa$ in   \eqref{lrc}.
Further, the map
$\eta$ is induced by the augmentation $a: \Ut\o\Ut\to\C$
and by the natural projection $pr:\ \Ug\o\Ug\onto\Ug\o_\FZ\Ug$;
the map $\eps$ is the imbedding
by $h\mto 1\o h\o1$.
Thus, the image of the map $\eps$ is contained in the center
of the algebra $\BU_c$.

Below, we use the notation introduced at
the beginning of \S\ref{Xsec}
 and observe that any $\BU_c$-module may be viewed 
as an
$\Ut\o\Ut$-module, via $\eps$.
\vskip 2pt

\npb{Given $\la\times\nu\in\bt^*\times\bt^*,$ let  ${\mathbf
M}\BU_{\la,\nu,c}$ be the full abelian  subcategory  of
$D^b\big(\Lmod{\BU_c}\big)$
whose objects
are $\BU_c$-modules $M$ such that
 one has
$M=M^{(I_{\la,\nu})}$,
resp. $\dul_{\la,\nu,c}$ be the  full triangulated subcategory
of  $D^b\big(\Lmod{\BU_c}\big)$ whose objects
are complexes $M$, of $\BU_c$-modules,  such that
 one has
$H^\hdot(M)=[H^\hdot(M)]^{(I_{\la,\nu})}$.}
\smallskip

For any object $M\in D^b(\Lmod{\BU_c})$ such that 
the $\Ut\o\Ut$-action on $H^\hdot(M)$ is locally
finite, there is a canonical
direct sum decomposition $M=\bigoplus_{\la\times\nu\in\bt^*\times\bt^*}
M^{(\la,\nu)}$ such that $M^{(\la,\nu)}\in\dul_{\la,\nu,c}$ for all $\la,\nu$.

Similarly, 
the composite $j\ccirc (\kappa_{x,y}\o\Id)\ccirc\eps$, in diagram \eqref{ucc},
 makes any $\D_{x,y}^c$-module
an $\Ut\o\Ut$-module.
Using this, one shows that for any 
cosets $\bar\la,\bar\nu\in\bt^*/\bt^*_\Z$ such that
$(\bar\la,\bar\nu)\in \bar\bt_{x,y}^\perp$ and any
monodromic complex $\CV\in
\Don_{\bar\la,\bar\nu,c}(\yy_{1,1}),$ there is a  canonical 
`derived' spectral decomposition: 
\beq{rg_spec}
 R\Ga(\yy_{x,y},\CV)=\underset{(\la,\nu)\,\operatorname{mod}(\bt^*_\Z\oplus\bt^*_\Z)=
(\bar\la,\bar\nu)}\bop\
 R\Ga(\yy_{x,y},\CV)^{(\la,\nu)}
\quad\text{where}\quad R\Ga(\yy_{x,y},\CV)^{(\la,\nu)}\,\in\, \dul_{\la,\nu,c}.
\eeq

Now, recall the maps $p,q$ from diagram \eqref{maps}.
The statement of part (i) of the following theorem
is a straightforward generalization 
of a result due to Hotta-Kashiwara, \cite{HK}, Theorem 1.

\begin{theorem}\label{hk} 
Let $\lambda\in\bt^*$ be a dominant regular  weight,
let $\bar\la$  be the image of
$\la$ under the projection $\bt^*\onto\bt^*/\bt^*_\Z$,
resp. $\th$ be the image of
$\la$ under the projection $\bt^*\onto\bt^*/W$, and $\Th$ be the image of
$\th$ under the projection $\bt^*/W\onto\bt^*/\waff.$
Then, for any  $c\neq -1,-2,\ldots,1-n$, we have
\vskip 2pt

\vi The following  diagram of functors commutes
$$
\xymatrix{
&&\Don_{\bar\la,-\bar\la,c}(\yy_{1,1})\ar[drr]^<>(0.5){\quad p_*q^![-\dim\CB]}
\ar[dll]_<>(0.5){\CV\mto R\Gamma(\yy,\CV)^{(\la,-\la-2\rho)}\quad}&&\\
\dul_{\la,-\la-2\rho,c}
\ar[rrrr]^<>(0.5){M\mto \dxc\stackrel{L}{\bigotimes}_{\BU_c}\, M}
&&&& \dcatchi.
}
$$
\vskip 1pt

\vii There is an isomorphism of functors 
$\dcatchi\to  D^b(\Lmod{\BU_c})$:
$$\big[R\Gamma(\yy_{1,w_0},\
\HC(-))\big]^{(\la,\la^\dag)}\
\cong \
[R\Gamma(\X,-)]^{(\th)}.$$
\end{theorem}

\subsection{Proof of Theorem \ref{main} and  Theorem \ref{hk}}
We are going to
use one general property of the convolution functor
$(-)* \CR_!^{\la,\la^\dag}$,
see \eqref{rconv}, that  can be deduced from
\cite{BB}, Theorem~12.

Let $x,y\in W$ and $\nu,\la\in\bt^*$.
The property says that, 
  in the bounded derived category
of $\Ug\o\Ug\o\Ug$-modules,
one has a  canonical  quasi-isomorphism
\beq{ikey}
[\RR\Gamma(\yy_{x,y}, \varpi\o\CV)]^{(\nu,-\la-2\rho)}\iso
\big[\RR\Gamma(\yy_{x,w_0y},\ \CV\, *\,
\CR_!^{\la,\la^\dag})\big]^{(\nu,\la^\dag)},
\quad\forall \CV\in\Don_{\bar\nu,-\bar\la,c}(\yy_{x,y}).
\eeq

\begin{proof}[Proof of Theorem \ref{hk}]
The proof of part (i) 
 is identical
to the proof of \cite{HK}, Theorem 1, and will be omitted.
We now  prove part (ii) of Theorem \ref{hk}. It will be convenient
to
introduce the following notation:
$\hhc(-):=q_!p^*(-)[\dim\CB]$, resp. $\chh(-):=p_*q^!(-)[-\dim\CB].$
These functors form an adjoint pair.

We prove first  an  auxiliary result saying that:

{\em There is an isomorphism between the following two functors}
$\Don_{\bar\la,-\bar\la,c}(\yy_{1,1})\to D^b(\Lmod{\BU_c})$:
\beq{hhcc}\big[R\Gamma(\yy_{1,1},\
\hhc(-))\big]^{(\la,-\la-2\rho)}\
\cong \
[R\Gamma(\X,-)]^{(\th)}.\eeq

To prove this, observe that
our assumptions on $\la$ and $c$ insure that one has mutually
quasi-inverse 
Beilinson-Bernstein   triangulated equivalences
\beq{BB}
\xymatrix{
\dul_{\la,-\la-2\rho,c}{\hphantom{x}}
\ar@<0.5ex>[rrr]^<>(0.5){\,\D^c_{1,1}{\stackrel{L}\o}_{\BU_c}\ (-)\,}&&&
{\hphantom{x}}
\Don_{\bar\la,-\bar\la,c}(\yy_{1,1})\,.\ar@<0.5ex>[lll]^<>(0.5){\,[\RR\Gamma(\yy_{1,1},-)]^{(\la,-\la-2\rho)}\,}
}
\eeq

Let
$\CF\in \dcatchi$ and $ M\in \dul_{\la,-\la-2\rho,c}$.
Put  $F:=\RR\Gamma(\X,\CF)$ and $\CM:=\D^c_{1,1}{\stackrel{L}\o}_{\BU_c} M\
\in \Don_{\bar\la,-\bar\la,c}(\yy_{1,1}).$ 
We compute

 \allowdisplaybreaks\begin{align*}
\RR\Hom_{\BU_c}\big(M,\ [\RR\Gamma(\yy_{1,1}, \hhc(\CF))]^{(\la,-\la-2\rho)}\big)&=
\RR\Hom_{\D^c_{1,1}}(\CM,\ \hhc(\CF))&\text{by \eqref{BB}}\\
&=\RR\Hom_{\dxc}(\chh(\CM),\ \CF)&\text{by adjunction}\\
&=\RR\Hom_{\dxc}(\dxc
{\stackrel{L}\o}_{\BU_c} M^{(\la,-\la-2\rho)},\ \CF)&\text{by part (i)}\\
&=\RR\Hom_{\BU_c}(M^{(\la,-\la-2\rho)},\ \CF)\\
&=\RR\Hom_{\BU_c}(M^{(\la,-\la-2\rho)},\ F)\\
&=
\RR\Hom_{\BU_c}(M,\ F)=
\RR\Hom_{\BU_c}(M,\ F^{(\th)}),
\end{align*}
where the last two equalities hold since 
$M=M^{(\la,-\la-2\rho)}$.

Thus, we have  established, for any $\CF\in \dcatchi,\ M\in 
\dul_{\la,-\la-2\rho,c},$ a functorial isomorphism:
$$
\RR\Hom_{\BU_c}\big(M,\ \RR\Gamma(\yy_{1,1},
 \hhc(\CF))^{(\la)}\big)=\RR\Hom_{\BU_c}\big(M,\
\RR\Gamma(\X,\CF)^{(\th)}\big).
$$

Such an isomorphism clearly yields an isomorphism of functors 
claimed in \eqref{hhcc}.

To complete the poof of part (ii) of Theorem \ref{hk}(ii), 
 we combine \eqref{hhcc} with the  quasi-isomorphism
\eqref{ikey} for $\CV:=\hhc(\CF)$. This way, using the definition
of the functor $\hc$, we obtain the 
isomorphism of functors stated in part (ii) of the theorem.
\end{proof}

\begin{proof}[Proof of Theorem \ref{main}] We have that
 both $\la$ and $\la^\dag$ are dominant
regular weights and $c\neq -1,-2,\ldots,1-n$.
Therefore, according to the Beilinson-Bernstein theorem,
each of the two functors $\Gamma(\yy_{1,w_0},-)^{(\la,\la^\dag)}:\
\Mon_{\bar\la,\bar{\la}^\dag,c}\to{\mathbf M}\BU_{\la,\la^\dag,c}$
and $\Gamma(\X,-):\ \Lmod{\dxc}\to\Lmod{\D(\X,c)}$
is exact and yields an equivalence of
{\em abelian} categories.
Also, the functor $F\mto F^{(\th)}$ is clearly exact
on the category of $\mathfrak Z$-locally finite modules.

Thus, the isomorphism of functors in Theorem \ref{hk}(ii)
implies that 
$$\big[\Gamma(\yy_{1,w_0},\
\HC(-))\big]^{(\la,\la^\dag)}:\ {\scr C}_{\Th,c}(\X)\to \dul_{\la,\la^\dag,c}$$
is an exact functor. Furthermore, this functor
is a composite of the functor $\HC$ and the
functor $[\Gamma(\yy_{1,w_0},
-)]^{(\la,\la^\dag)}$, which is an equivalence
of the corresponding  abelian categories.

It follows that $\HC$ must itself induce an exact functor
between  abelian categories.
\end{proof}

\section{Mirabolic Harish-Chandra $\D$-module}
\subsection{}
 There is  an especially
important family of  mirabolic character $\D$-modules that was introduced in
 \cite[\S ~7.4]{GG}. To define these $\D$-modules,
recall the map \eqref{lrc} and the subsets
$\FZ_\th^c=\upsilon\zeta(\FZ_\th)\sset \D(\X,c)$, see \S\ref{Xsec}.

\begin{definition}\label{hcmod}
 For any  $(\th, c)\in (\bt^*/W)\times\C,$
we  define a $\dxc$-module called {\em Harish-Chandra $\dxc$-module},
resp. a projective limit of $\dxc$-modules
called {\em generalized}  Harish-Chandra $\D$-module, as follows
$$
\CG^{\th,c}:=\dxc/\big(\dxc\ \g+
\dxc\ \FZ^c_\th\big),
\quad\oper{resp.}\quad\wh\CG^{\th,c}:=\underset{^{m\to\infty}}{\oper{lim\,proj}}\ \dxc/\big(\dxc\ \g+
\dxc\ (\FZ^c_\th)^m\big).
$$
\end{definition}

It is clear that we have $\CG^{\th,c}\in{\scr C}_{\Th,c}$,
where $\Th\in \bt^*/\waff$ is the image of $\th$
under the projection $\bt^*/W\onto\bt^*/\waff$.
Therefore, Proposition \ref{hc_prop} implies that
$\CG^{\th,c}$ is a regular holonomic $\D$-module.

Our goal is to provide a geometric construction of
the  Harish-Chandra $\D$-module $
\CG^{\th,c}$ similar to one
given in \cite{HK1}  in the
classical case.

To this end, let $U$ be
the  unique
Zariski open and dense $G\times \BT\times\BT$-orbit
in $\yy_{1,w_0}$. 
One may mimic definitions of the local systems
${\mathcal J}^\la$ and $\wh{\mathcal J}^\la,$
on ${\mathcal U}$, see \S\ref{R},
and introduce analogous local systems,
more precisely, monodromic $\D^c_U$-modules $\lll^{\la,c}$ and
$\wh{\lll}^{\la,c}$.

Write $j: U\into\yy_{1,w_0}$
for the open imbedding. 

\begin{theorem}\label{kthm}
  Let $\lambda\in\bt^*$ be a
 sufficiently
dominant regular weight, and
let  $\th$ be the image of
$\la$ under the projection $\bt^*\onto\bt^*/W$. Then,  for 
 sufficiently large real  $c\gg 0$,
in $\dcat$ there is an 
isomorphism
\beq{hc_module}
\CG^{\th,c}\
\cong\ \ch(j_!{\lll}^{\la,c}),\quad\oper{resp.}\quad
\wh\CG^{\th,c}\
\cong\ \ch(j_!{\wh{\lll}}^{\la,c}).
\eeq
\end{theorem}

The theorem
 provides
a purely geometric construction of the
perverse sheaf that corresponds to  the Harish-Chandra
$\D$-module $\CG^{\th,c}$
via the Riemann-Hilbert correspondence.

The proof of Theorem \ref{kthm} will occupy sections
\ref{duality_sec}-\ref{hthpf}.

\begin{corollary}\label{proproj} 
With the assumptions of Theorem \ref{main}, we have that
$\wh\CG^{\th,c}$  is a projective (pro)-object of the category ${\scr
C}_c(\X).$
\end{corollary}

\begin{proof}[Proof of Corollary] In general, let
${P}$ be a projective (pro)-object of the
abelian category ${{\mathsf{Mon}}_c(\yy_{1,w_0})}$ such that the complex
$\ch({P})$ is concentrated in degree zero.

We claim that  $\Hom_{\dxc}(\ch({P}),\ -)$
is an exact
functor on the category ${\scr C}_c(\X)$. To check this,
we use  adjunction and obtain
$$\Hom_{\dxc}(\ch({P}),\ \CM)=\Hom_{\D^c_\yy}({P},\ \hc(\CM)),
\qquad\forall\CM\in  {\scr C}_c(\X).
$$

Here,  the functor $\CM\mto \hc(\CM)$ is exact by Theorem
\ref{main}, and the functor
$\Hom_{\D^c_\yy}({P},-)$ is exact since ${P}$ is projective.
Thus, we have proved our claim. We conclude that
$\ch({P})$ is a projective (pro)-object of the category ${\scr
C}_c(\X).$

To complete the proof of the corollary,
 we observe that ${\wh{\lll}}^{\lambda,c}$ is a projective (pro)-object
in the category of $G$-equivariant local systems on  $U$, the
open $G\times\BT\times\BT$-orbit.
Therefore, $j_!{\wh{\lll}}^{\lambda,c}$ 
is a projective (pro)-object
of the category ${{\mathsf{Mon}}_c(\yy_{1,w_0})}$. The result  follows.
\end{proof}

\subsection{}
\label{duality_sec} We begin with a general setting.

Let $X$ be a smooth variety and
let $\D$ be a TDO on $X$. Let ${\mathfrak G}$ be an  $m$-dimensional
filtered Lie algebra equipped with a filtration preserving Lie algebra morphism
${\mathfrak G}\to \D,\,u\mto\arr{u},$ into
 (not necessarily first order) twisted
differential operators. 
Let $\D{\mathfrak G}$
be the left ideal in $\D$ generated by the image of ${\mathfrak G}$.

We recall the following version
of a result due to G. Schwarz \cite{Sc}, \S8, and  M. Holland, \cite{H},
Proposition 2.4,
independently.

\begin{lemma}\label{holl} 
Assume that 
the image of $\gr{\mathfrak G}\to\gr\D,$   the  associated graded
morphism,
 gives
a regular sequence  in $\gr\D$.
Then, the natural map
\begin{equation}\label{holl_iso}\gr\D/\gr\D\gr{\mathfrak G}\onto \gr(\D/\D{\mathfrak G}),
\end{equation}
is a bijection. Moreover, the standard Chevalley-Eilenberg complex
associated with  the action of the Lie algebra ${\mathfrak G}$
 on $\D$ by right multiplication,
\begin{equation}
\label{Koszul}
0\to\D\o\LLa^m{\mathfrak G}\to\D\o\LLa^{m-1}{\mathfrak G}\to\ldots\to
\D\o\LLa^2{\mathfrak G}\to\D\o{\mathfrak G}\to\D\to0,
\end{equation}
is a free $\D$-module resolution of $\D/\D{\mathfrak G}$, a left $\D$-module.
\end{lemma}

The algebra morphism $\gr{\mathfrak G}\to\gr\D$ is induced by a moment map
$\mu: T^*X\to (\gr{\mathfrak G})^*$.
The condition of the lemma that the image of $\gr{\mathfrak G}$ be a
regular sequence may be reformulated as a requirement that
 the Koszul complex
\begin{equation}
\label{rezolventa}
0\to\CO_{T^*X}\otimes\LLa^m(\gr{\mathfrak G})
\to\CO_{T^*X}\otimes\LLa^{m-1}(\gr{\mathfrak G})\to\ldots\to\CO_{T^*X}\otimes\LLa^1(\gr{\mathfrak G})\onto
\CO_{T^*X}\to0,
\end{equation}
be a resolution of the  structure
sheaf $\CO_{\bmu^{-1}(0)},$ by free $\CO_{T^*X}$-modules.
 The latter condition 
is also equivalent to the  condition that
$\mu\inv(0),$ the scheme-theoretic zero
fiber of the map $\mu$, be a complete intersection.
\subsection{}\label{RGsec}
Let $\la,\nu\in\bt^*$.
We may view the pair $(\la,\nu)$ as a  linear function
$\la\times\nu:\ \bt\oplus\bt\to\C$ and write
$\bt^{\la,\nu}\sset \D_{x,y}^c$ for the
subspace associated with the $\BT\times\BT$-action on
$\yy_{x,y}$ and with the character $\la\times\nu$, as 
explained in \S\ref{psi}. Let
$\D_{x,y}^c\bt^{\la,\nu}\sset \D_{x,y}^c$ be the corresponding left
ideal.

We observe  that one actually has an equality
$\D_{x,y}^c\bt^{\la,\nu}=\D_{x,y}^c$ unless
the linear function $\la\times\nu$ annihilates
the subspace $\bt_{x,y}\sset \bt\oplus\bt,$ i.e.
unless we have $\dis\la\times\nu\,\in\, \bt_{x,y}^\perp\sset \bt^*\oplus\bt^*$.
It is clear that the latter holds iff there exists
an element $\gamma\in\bt^*$ such that
$\la=x(\gamma)$ and $\nu=-y(\gamma).$

For any $c\in\C$ and  $x,y\in W,$  we introduce
a family of left $\D_{x,y}^c$-modules
\beq{CV}\CV^{\la,\nu,c}_{x,y}:=
\D_{x,y}^c\big/(\D_{x,y}^c\,\g+
\D_{x,y}^c\,\bt^{\la,\nu})\ \in\,\Mon_c(\yy_{x,y}),
\qquad\la\times\nu\in \bt_{x,y}^\perp.
\eeq
It is clear that   $\CV^{\la,\nu,c}_{x,y}$ is a $G$-equivariant
 regular holonomic $\D$-module, by
Corollary \ref{zwei}(i).
\medskip

Next,  put $D_{x,y}:=\Gamma(\yy_{x,y},\
\D^c_{x,y}).$

 \begin{lemma}\label{RGlem} For any
$(\la,\nu)\in\bt_{x,y}^\perp,$ one has:
$\dis\ R\Gamma(\yy_{x,y},\ \CJ^{\la,\nu,c}_{x,y})\simeq
D_{x,y}/(D_{x,y}{}^{\,}\g+D_{x,y}{}^{\,}\bt^{\la,\nu}).$ 
\end{lemma}

\begin{proof}
We  use  the Koszul complex
\eqref{Koszul} in the case
where $X=\yy_{x,y}.$ The Koszul
 complex is acyclic, by Proposition \ref{ein moment}(ii) and Lemma \ref{holl}.
It is well known that one has
$R^\ell\Gamma(\yy_{x,y},\
\D^c_{x,y})=0,$
for any $\ell>0$, since this is the case for the
associated graded algebra, thanks to the
Grauert-Riemenschnider theorem.
Thus,  the Koszul complex
\eqref{Koszul} provides a {\em left} $\Gamma$-acyclic resolution 
of $\CJ^{\la,\nu,c}_{x,y}$. Computing
$R\Gamma(\yy_{x,y},\ \CJ^{\la,\nu,c}_{x,y})$ via this resolution
yields the result.
\end{proof}

\begin{corollary}
\label{cjj}  
Assume $\la\in\bt^*$ is regular, and $c\ne-1,-2,\ldots,1-n$.  
We have isomorphisms:
$$\CJ^{\lambda,-\lambda,c}_{1,1}\ \simeq\
\CJ^{\lambda,\lambda^\dag,c}_{1,w_0}\,
*\,\CR^{-\la^\dag,-\la}_*,
\quad\operatorname{resp.}\quad
\CJ^{-\lambda,\lambda,-c}_{1,1}\ \simeq \
\CJ^{-\lambda,-\la^\dag,-c}_{1,w_0}\,*\,\CR^{\la^\dag,\la}_!.
$$
\end{corollary}

\begin{proof} We only prove the second isomorphism,
the proof of the other one being similar. 
To do this, we first verify, by a simple direct calculation
that
one has a natural isomorphism of $\Ug\o\Ug\o\Ug$-modules
$$D^{-c}_{1,1}/(D^{-c}_{1,1}{}^{\,}\g+D^{-c}_{1,1}{}^{\,}\bt^{-\la,\la})
\cong
D^{-c}_{1,w_0}/(D^{-c}_{1,w_0}{}^{\,}\g+D^{-c}_{1,w_0}{}^{\,}\bt^{-\la,-\la^\dag}).
$$

Hence, using Lemma \ref{RGlem} we may rewrite
the above isomorphism
as follows
$$[R\Gamma(\yy_{1,1},\ \CJ^{-\la,\la,-c}_{1,1})]^{(-\la,\la)}
\cong
[R\Gamma(\yy_{1,w_0},\ \CJ^{-\la,-\la^\dag,-c}_{1,w_0})]^{(-\la,-\la^\dag)}.$$

Further, by \eqref{ikey}, the object
 on the right hand side above is
isomorphic, in the bounded derived category of
$\Ug\o\Ug\o\Ug$-modules, to
$[R\Gamma(\yy_{1,1},\ \CJ^{-\lambda,-\lambda^\dag,-c}_{1,w_0}\,
*\,\CR^{\la^\dag,\la}_!)]^{(-\la,\la)}$.

 Therefore,
the functor $[R\Gamma(\yy_{1,1},-)]^{(-\la,\la)}$
takes $\D$-modules
$\CJ^{-\la,\la,-c}$ and $\CJ^{-\lambda,-\lambda^\dag,-c}_{1,w_0}\,
*\,\CR^{\la^\dag,\la}_!$ to isomorphic objects.
The result now follows from the Beilinson-Bernstein theorem.
\end{proof}

\subsection{}
We keep the setting of section \ref{hor}
and let the Weyl group act on $\bt^*$ via the
dot-action. We use the identification $\Spec\FZ\cong\bt^*/W,$ provided
by the Harish-Chandra isomorphism. 
 
The canonical isomorphism $\tau: \Ug\to(\Ug)^\opp$
restricts to an automorphism $\tau: \FZ\to\FZ$. Write
$\tau: \Spec\FZ\to\Spec\FZ$ for the induced automorphism.
In terms of  the Harish-Chandra isomorphism
$\Spec\FZ\cong\bt^*/W,$ one can write
$\tau(W\cdot \la)=W\cdot(-\la-2\rho).$

Recall the notation from Diagram \eqref{maps}.

\begin{proposition}\label{kwa}
Assume $\la\in\bt^*$ is regular, and $c\ne-1,-2,\ldots,1-n$.
Write $\th\in\bt^*/W$ for the image of $\la$.  Then,
we have an isomorphism
\beq{cjjf}
\CG^{\th,c}=p_*q^!(\CJ^{\la,-\la,c}_{1,1})[-\dim\CB].
\eeq
\end{proposition}

\begin{proof}
 Given $\th\in\Spec\FZ,$ we introduce the notation
$\CU_\th:=\Ug/\Ug\cdot\FZ_\th$. Further, write
 $\CU_{\th,\tau(\th),c}:=\CU_\th\otimes\CU_{\tau(\th)}\otimes\CU_c$. We
introduce a coproduct $\Delta$, an algebra map defined on generators as follows
$$\Delta:\
\Ug\too (\Ug)^{\otimes3}\onto
\CU_{\th,\tau(\th),c},\quad
x\mapsto1\otimes1\otimes x+1\otimes x\otimes1+x\otimes1\otimes1.
$$

Clearly, we have $\CG^{\th,c}=\D_{G,\th}
\otimes_{\CU_{\th,\tau(\th)}}
\left(\CU_{\th,\tau(\th),c}/
\CU_{\th,\tau(\th),c}
{}^{\,}\Delta(\g)\right)$,
where  $\D_{G,\th}:=\D_G/\D_G\,\lr(\FZ_\th\o1).$
Observe further that the $\D$-module $\CG^{\th,c}$ is the top (zeroth) cohomology module of the 
complex
\begin{equation}\label{RG}{}
^L\CG^{\th,c}:=\D_{G,\th}
\stackrel{L}{\otimes}_{\CU_{\th,\tau(\th)}}
\left(\CU_{\th,\tau(\th),c}/
\CU_{\th,\tau(\th),c}
{}^{\,}\Delta(\g)\right).
\end{equation}

To complete the proof, we must show, in view of Theorem ~\ref{hk}(i), that
the canonical morphism $\dis ^L\CG^{\th,c}\to\CG^{\th,c}$
is a quasi-isomorphism.

To this end, recall that the canonical line bundle $\omega_\CB$ is isomorphic 
to the line bundle $\CO_\CB(-2\rho)$.
Thus, we must prove
$R\Gamma(\CB\times\CB\times\P,\CJ^{\la,-\la-2\rho,c}_{1,1})=
\CU_{\th,\tau(\th),c}/
\CU_{\th,\tau(\th),c}{}^{\,}\Delta(\g)$.
But this last isomorphism holds by
Lemma \ref{RGlem}, and we are done.
\end{proof}

\subsection{} We are going to
reduce  the proof of Theorem 
\ref{kthm} to Theorem \ref{bthm} of Section 2.

To this end, recall that associated with any
integral weight $\nu\in\BT^*_\Z$, there
is a $G$-equivariant line bundle $\oo(\nu)$ on $\CB$.
We put $\lll:=\CO_{\tB_{1,w_0}}\boxtimes\CO_\P(n),$
a $G\times\BT\times\BT$-equivariant line bundle
on $\yy_{1,w_0}$.
Further,
we introduce a character
 $\phi:=(0,2\rho,2\rho)\in\g^*\times\bt_{1,w_0}^\perp\sset\g^*\times\bt^*\times\bt^*.$ 

The following result  shows, in particular,
that the open imbedding $j: U\into \yy_{1,w_0}$ is an affine  morphism.

\begin{lemma}\label{s}  
There is a $\phi$-semi-invariant
(with respect to the
$G\times\BT\times\BT$-action)
section $s\in \Gamma(\yy_{1,w_0}, \lll),$
such that one has $\ s\inv(0)=\yy_{1,w_0}\sminus U.$
\end{lemma}
\begin{proof}
 Let $V=\C^n$ and let $\omega_i,\ i=1,\ldots,n-1$,
stand for the fundamental weights of $G$. We will write $\omega_0=0$.
Fix a volume functional $\vol: \LLa^nV\to\C$.

The line bundle
$\oo({\omega_1})$ on $\CB$ descends to the
line bundle $\CO_\P(1)$ on $\P$, 
and the space of its global
sections is canonically isomorphic to $V^*$. Let $s_i$ stand for the global section of 
$\oo({\omega_i})\boxtimes\oo({\omega_{n-i}})$ which sends $v_i\otimes
v_{n-i}\in \LLa^iV\otimes\LLa^{n-i}V=\Gamma[\CB\times\CB,
\oo({\omega_i})\boxtimes\oo({\omega_{n-i}})]^*$ to the volume of
$v_i\wedge v_{n-i}$.
Let $s_i$ be its lift to a global section of 
$\oo({\omega_i})\boxtimes\oo({\omega_{n-i}})\boxtimes\CO_\P$.

Let $\varsigma_j,\,0\leq j\leq n-1,$ be the global section
of $\oo({\omega_j})\boxtimes\oo({\omega_{n-1-j}})\boxtimes\CO_\P(1)$
such that, for any
$$v_j\otimes v_{n-1-j}\otimes
v\in\LLa^jV\otimes\LLa^{n-1-j}V\otimes
V=\Gamma(\CB\times\CB\times\P,
\oo({\omega_j})\boxtimes\oo({\omega_{n-1-j}})\boxtimes\CO_\P(1))^*,
$$
we have that $\langle \varsigma_j,v_j\otimes v_{n-1-j}\otimes
v\rangle=\vol(v_j\wedge v_{n-1-j}\wedge v).$

 Finally, we denote by $s$ the
global section $s_1\ldots s_{n-1}\varsigma_0\ldots\varsigma_{n-1}$ of
the product $\oo({2\rho})\boxtimes\oo({2\rho})\boxtimes\CO(n)\simeq
\omega^{-1}_{\CB\times\CB\times\P}$ of the
above line bundles on $\CB\times\CB\times\P$.

Now, we use an explicit classification of $G$-diagonal orbits in
$\CB\times\CB\times\P$ given in \cite[\S 2.11]{MWZ}.
The  classification shows
that any codimension one $G$-orbit
in $\yy_{1,w_0}$ is equal (locally) to  the zero locus
of either  one of the sections $s_i$ or of one of the sections
$\varsigma_j$. One deduces that 
 the set $(\CB\times\CB\times\P)\sminus
s^{-1}(0)$
is the open $G$-diagonal orbit in $\CB\times\CB\times\P$.
The statement of the lemma easily follows from this.
\end{proof}

\begin{remark}\label{masaki} It is likely that there is an explicit closed expression,
as a product of linear factors,
for the $b$-function associated with the section $s$ of Lemma \ref{s}.
 We expect that such an expression may be obtained
by adapting  arguments used by M. Kashiwara in \cite{K3}. The knowledge
of the roots of the $b$-function  would make it possible to give
explicit sharp bounds on the parameter `$c$' which are necessary and sufficient 
for the statement of Theorem \ref{kthm}
to hold true.
\end{remark}

\begin{lemma}\label{tr} Let $\lll$ and $\phi$ be as in Lemma
\ref{s}. For any $x\in \yy_{1,w_0}\sminus U$,
using the notation of formula \eqref{match},
we have  $\chi_{\lll,x} \neq \phi|_{\g_x}$.
\end{lemma}
\begin{proof} Fix a point  $x\in \yy_{1,w_0}\sminus U$,
and let $G_x$ be its isotropy group in $G$.
Travkin has shown in \cite[Lemma 7]{T} that there
exists
a 1-parameter subgroup $\gamma:\C^\times\to G_x$
and an integer $m>0$
such that, for any $z\in \C^\times,$ one has
$\chi_{\lll,x}\ccirc\gamma(z)=z^m$.
On the other hand, since
$\phi:=(0,2\rho,2\rho)$, we have $\phi \ccirc\gamma(z)=1$.
We conclude that $\chi_{\lll,x}\ccirc\gamma \neq \phi\ccirc\gamma$,
hence $\chi_{\lll,x}\neq \phi|_{\g_x}$.
\end{proof}

\subsection{Proof of Theorem \ref{kthm}}\label{hthpf}
For any $\la\in\bt^*$, the sheaf
$\lll^{\la,c}:=j^*\CJ^{\lambda,\lambda^\dag,c}_{1,w_0}$
is
 a locally free rank one $G$-equivariant $\CO_U$-module, i.e. a line
bundle on ~$U$.

\begin{lemma}\label{CJJ}
For dominant
 enough $\lambda,$ and $c\gg0$, the canonical
morphisms below induce isomorphisms
$$
\CJ^{\lambda,\lambda^\dag,c}_{1,w_0}\iso 
j_*j^*\CJ^{\lambda,\lambda^\dag,c}_{1,w_0},
\quad\operatorname{resp.}\quad
j_!j^!\CJ^{\lambda,\lambda^\dag,c}_{1,w_0}\iso
\CJ^{\lambda,\lambda^\dag,c}_{1,w_0}.
$$
\end{lemma}

\begin{proof} We apply
Theorem~\ref{bthm}
to  $X:=\yy_{1,w_0}$ and to the section $s$ from Lemma \ref{s}.
Condition \eqref{match} of Theorem~\ref{bthm}
holds thanks to Lemma \ref{tr}.
Thus,  Theorem~\ref{bthm} says that, given $\lambda\in\bt^*,\
c\in\C,$ for all integers $k\gg0$, one has
$$\CJ^{-\lambda-k\rho,\ -\la-k\rho,\ -c-n}_{1,w_0}\simeq
j_{*}\lll^{-\lambda-k\rho,\ -\la-k\rho,\ -c-kn}_{1,w_0},
\quad\operatorname{resp.}\quad
\CJ^{\lambda,\ \la^\dag+k\rho,\ c+kn}_{1,w_0}\simeq j_!\lll^{\lambda,\ \la^\dag+k\rho,\ c+kn}_{1,w_0}.
$$
This proves the lemma.
\end{proof}
\smallskip

We can now complete the proof of Theorem \ref{kthm}.
To this end, we combine
 Corollary \ref{cjj} with Lemma \ref{CJJ}.
We deduce that, for all sufficiently dominant
 enough $\lambda\in\bt^*$ and $c\gg0$,
one has
$$
\CJ^{-\lambda,\lambda,-c}_{1,1}\ \simeq \
(j_*\lll^{-\lambda,-\la^\dag,-c}_{1,w_0})*\CR^{\la^\dag,\la}_!,
\quad\operatorname{resp.}\quad
\CJ^{\lambda,-\lambda,c}_{1,1}\ \simeq\
(j_!\lll^{\lambda,\lambda^\dag,c}_{1,w_0})
*\CR^{-\la^\dag,-\la}_*.
$$

Thus, using the isomorphism in \eqref{cjjf}, we obtain
$$
\CG^{\th,c}=p_*q^!(\CJ^{\la,-\la,c}_{1,1})[-\dim\CB]
=p_*q^!(\varpi\o
(j_*\lll^{-\lambda,-\la^\dag,-c}_{1,w_0}*\CR^{\la^\dag,\la}_!))
=\ch(j_*\lll^{-\lambda,-\la^\dag,-c}_{1,w_0}),
$$
and the theorem is proved.\qed

\section{Further properties of the  mirabolic Harish-Chandra
$D$-module}

\subsection{}\label{le} Let $X$ be a smooth manifold and
let $\Delta:\ X\to X\times X$ be the diagonal embedding. 
Given  a class
$\chi\in H^2(X, \Om^{1,2}_X)$, we put
$\D_{\chi,-\chi}:=\D_\chi\boxtimes\D_{-\chi}$, a TDO  on $X\times X$.
In the notation
of ~\cite{K}, 2.8, we have that $\Delta^\sharp\D_{\chi,-\chi}$ is the sheaf
of  {\em non-twisted} differential operators on $X$. 
The category 
$\D_\chi\dash\module\dash\D_\chi,$ of $\D_\chi\dash\D_\chi$-bimodules,
is equivalent
to the category of $\D_\chi\boxtimes\D_{\omega-\chi}$-modules.

The pushforward
 $\D_{\chi,-\chi}$-module, $\Delta_*\CO_X$, viewed as a left $\D_\chi$-module,
is canonically isomorphic to $\D_\chi\otimes_{\CO_X}\omega^{-1}_X$, 
 see ~\cite{K}, 2.11. On the other hand, $\D_\chi$,
the diagonal $\D_\chi\dash\D_\chi$-bimodule
viewed as a left $\D_\chi\boxtimes\D_{\omega-\chi}$-module,
is canonically isomorphic to 
$\Delta_*\CO_X\otimes \pr_2^*\omega_X$ where $\pr_2: X\times X\to X$ is the
 second projection.

Given a holonomic $\D_\chi$-module $\CF$, the complex 
${R{\scr H}_{\!}om}_{\D_\chi}(\CF,\D_\chi)$ 
has the only cohomology in degree $d=\dim X$, so it is quasi-isomorphic to 
${{\scr E}xt}{}^d_{\D_\chi}(\CF,\D_\chi)$.
The right action of $\D_\chi$ on itself gives rise to the right action
of $\D_\chi$ on ${{\scr E}xt}{}^{d}_{\D_\chi}(\CF,\D_\chi)$.
Thus we have a $\D_\chi^\opp=\D_{\omega-\chi}$-module
${{\scr E}xt}{}^{d}_{\D_\chi}(\CF,\D_\chi)$, and we define
the $\D_{-\chi}$-module $\BD(\CF):=\omega^{-1}_X\otimes_{\CO_X}
{{\scr E}xt}{}^{d}_{\D_\chi}(\CF,\D_\chi)$.

Fix a Lie algebra $\g$  of dimension $m$,
with {\em modular} character  $\delta(-)=\Tr\ad$.
We use the notation  of \S~\ref{psi}, and 
write $\mu: T^*X\to \g^*$ for the moment map.

\begin{lemma}\label{duality_lem} Assume that 
the moment map  $\mu: T^*X\to \g^*$ is flat. Then, we have
a natural isomorphism
$$R{{\scr H}_{\!}om}_\D(\D/\D{}^{\,}\g^\psi,\D)\cong 
\D^\opp/\D^\opp{}^{\,}\g^{\delta-\psi}[-\dim X].$$
\end{lemma}

\begin{proof} By assumptions, we may apply Lemma \ref{holl}
to the Lie algebra ${\mathfrak G}:=\g$, all placed in filtration degree
1,
and to the Lie algebra map $\g\to\D,\, u\mto \arr{u}-\psi(u)\cdot 1.$
We conclude that the corresponding complex \eqref{Koszul}
provides a free $\D$-module resolution of $\D/\D{}^{\,}\g^\psi$.

Clearly, we have
$${{\scr H}_{\!}om}_\D(\D\otimes\LLa^p\g,\D)\cong
\D^\opp\otimes\LLa^p\g^*\cong
\D^\opp\o\LLa^{m-p}\g\otimes\LLa^m\g^*,
\qquad \forall p=0,\ldots.$$
Thus, using resolution \eqref{Koszul}, we deduce
that the object
$R{{\scr H}_{\!}om}_\D(\D/\D{}^{\,}\g^\psi,\D)$ may be representated by the complex
$$
\D^\opp\otimes\LLa^m\g\otimes\LLa^m\g^*\to
\D^\opp\otimes\LLa^{m-1}\g\otimes\LLa^m\g^*\to\ldots\to
\D^\opp\otimes\g\otimes\LLa^m\g^*\to
\D^\opp\otimes\LLa^m\g^*.
$$

But this complex is acyclic in positive degrees
since the corresponding  associated graded complex is
nothing but the resolution \eqref{rezolventa},
tensored by $\LLa^m\g^*$. Further, the cokernel of
the rightmost map $\D^\opp\otimes\g\otimes\LLa^m\g^*\to
\D^\opp\otimes\LLa^m\g^*,$ in the above complex, is
equal to $\D^\opp/\D^\opp{}^{\,}\g^{\delta-\psi}.$
The result follows.
\end{proof}

\begin{corollary} For any $\lambda,\nu\in\bt^*$ and $c\in\C,$
there is a canonical isomorphism 
$$\BD(\CJ_{\lambda,\nu,c})\simeq\omega^{-1}_{\CB\times\CB\times\P}
\otimes_{\CO_{\CB\times\CB\times\P}}\CJ_{-\lambda-2\rho,-\nu-2\rho,-c-n}.
\eqno\Box$$
\end{corollary}

\subsection{Duality for the  Harish-Chandra $\D$-module.}\label{tri}
Recall the automorphism $\tau: \FZ\to\FZ$ induced
by the  isomorphism $\tau: \Ug\to(\Ug)^\opp$, see \S\ref{RGsec}.

\begin{proposition}
\label{twist} 
For any $(\th,c)\in\Spec\FZ\times\C,$ there is a canonical isomorphism
$$\BD(\CG^{\th,c})\simeq\omega^{-1}_\P\otimes_{\CO_\P}
\CG^{\tau(\th),-c-n};$$
\end{proposition}

\begin{proof}
We choose and fix a finite dimensional vector
subspace $\fH_\th\sset \FZ_\th$ that freely generates
the center, i.e., such the the imbedding $\fH\into \FZ_\th$
induces an algebra isomorphism $\Sym\fH_\th\iso\FZ.$
We will view the vector space $\fH_\th$ as an abelian
Lie subalgebra of the  enveloping algebra $\Ug$.

We put $\V:=\g\oplus \fH_\th$, and view this
direct sum as a filtered Lie algebra,
the direct sum of the Lie algebra $\g$, placed in filtration
degree 1, and $\fH_\th$, and abelian
Lie algebra equipped with filtration induced by the standard filtration on the 
 enveloping algebra $\Ug$.

We imbed $\Ug\into \D(G)$
as left-invariant differential operators.
This imbedding restricts to a filtration preserving 
map $\fH_\th\to\D(G)$. Combining this map with the
natural map $\g\to\D(\X,c)$, induced by the
$G$-action of the group $G$ on itself by left translations,
we get  a filtration preserving Lie algebra
map $\V=\g\oplus \fH_\th\to\D(\X,c)$.

\begin{claim}\label{cl}
The corresponding moment map
$\mu: T^*\X\to (\gr\V)^*$ is flat.
\end{claim}

This claim is a reformulation of \cite{GG}, Proposition 2.5;
the map denoted by $\bmu\times\boldsymbol{\pi}:
T^*\X\to \g\times\C^{(n)}$ in {\em loc cit}  is nothing
but the moment map
$\mu$, of the claim above.

By Claim \ref{cl}, we are in a position to apply Lemma \ref{holl}.
According to the latter, 
the complex
\begin{equation}\label{Kosz}
0\to\dxc\otimes\LLa^{\operatorname{top}}(\g\oplus\fH_\th)
\to\ldots\to
\dxc\otimes\LLa^1(\g\oplus\fH_\th)\to
\dxc\to 0,
\end{equation}
provides a resolution of the left $\dxc$-module
 $\CG^{\th,c}=\dxc/\dxc(\g\oplus\fH_\th).$

We now complete the proof of Proposition \ref{twist}(i).
The algebra  $\Gamma(\P,\D_{\P}^{c})$ of global sections  is the quotient
algebra $\CU_c,$ of $\Ug$. The anti-involution $\tau:\ \Ug\to(\Ug)^\opp$
induces an isomorphism $\tau:\ \CU_c^\opp\simeq\CU_{-n-c}$.
Recall that $\omega_\P=\CO_\P(-n)$. The canonical isomorphism
$(\D_{\P})^\opp\simeq\D_{\P,-n-c}$ coincides with $\tau$ at the level of
global sections. 

 We choose a left-invariant nonvanishing
top degree differential form $\beta$ on $G$. This form is also
right-invariant, and it trivializes the canonical bundle $\omega_{G}$.

The algebra $\D(G)$ is isomorphic to the smash product
$\BC[G]\ltimes\Ug$ (we embed $\Ug$ into $\D(G)$ as left-invariant
differential operators).
Using the trivialization of $\omega_{G}$ we get a canonical isomorphism $\D(G)^\opp\simeq\D(G)$.
For $h\in\BC[G]$, and $x\in\g\subset\Ug$, the action of the
anti-involution is described as $h\otimes1\mapsto h\otimes1,\
1\otimes x\mapsto-1\otimes x$. In particular, the anti-involution 
restricts to $\tau$ on $\Ug$. So we keep the name $\tau$ for the 
anti-involution of $\D(G)$. More generally, for any vector field
$ v$ on $G$ we have 
$\tau( v)=- v+\frac{{\operatorname{Lie}}_ v\beta}{\beta}$, 
cf.~\cite{K},~2.7.1. In particular, if $ v$ is a {\em right}-invariant
vector field, we have $\tau( v)=- v$. Moreover, the adjoint action of 
$G$ on itself gives rise to the embedding $\ad:\ \g\hookrightarrow
\D(G)$, and we have $\tau(\ad(x))=-\ad(x)=\ad(-x)$.

To compute 
$\BD(\CG^{\th,c})$, the dual of the Harish-Chandra module,
we  use the free 
$\dxc$-resolution ~\eqref{Kosz}.
Choose a trivialization
$\LLa^\Top(\g\oplus\fH_\th)\simeq\BC$ of the one dimensional vector
space $\LLa^\Top(\g\oplus\fH_\th)$. Hence we obtain a perfect
pairing 
$$\LLa^k(\g\oplus\fH_\th)\times\LLa^{n^2+n-2-k}(\g\oplus\fH)
\to\BC.
$$

We apply the functor
${\scr H}^{\!}om_{\dxc}(-,\dxc)$ to the resolution ~\eqref{Kosz}. Thus,
we see that the object
$R{\scr H}^{\!}om_{\dxc}(\CG^{\th,c},\dxc)$ is represented
by the following complex of {\em right}
$\dxc$-modules
$$
0\to\LLa^{\operatorname{top}}(\g\oplus\fH)\otimes\dxc
\to\ldots\to
\LLa^1(\g\oplus\fH)\otimes\dxc\to
\dxc
$$
arising from the action of $\g\oplus\fH\subset\dxc$
on $\dxc$ by the {\em left} multiplication.
The above complex is acyclic everywhere except 
the rightmost term, by Claim \ref{cl} again.

Combining this description of the right $\dxc$-module
${{\scr E}xt}{}^{n^2+n-2}_{\dxc}(\CG^{\th,c},
\dxc)$ with the above description of the isomorphism
$$\tau\boxtimes\tau:\ (\D_\X^c)^\opp=\D_G^\opp\boxtimes
(\D_{\P}^{c})^\opp\simeq\D_G\boxtimes\D_{\P,-n-c}=\D_{\X,-n-c}$$
we see that 
$$
(\tau\boxtimes\tau)\left({{\scr E}xt}{}^{n^2+n-2}_{\dxc}
(\CG^{\th,c},\dxc)\right)\simeq\CG^{\tau(\th),-n-c}
$$

We conclude that $\dis
\BD(\CG^{\th,c})\simeq\CO_\P(n)\otimes_{\CO_\P}
\CG^{\tau(\th),-n-c}.
$ The proposition is proved.
\end{proof}

From Claim \ref{cl}, using Lemma \ref{holl}, we deduce
\begin{corollary} For any $(\th,c)\in \Spec\FZ\times\C,$
the characteristic cycle $[SS(\CG^{\th,c})]$ equals $[\mu\inv(0)]$,
the
cycle of the scheme-theoretic zero fiber of the moment map $\mu: T^*\X\to (\gr\V)^*$.\qed
\end{corollary}

\subsection{Reminder on affine Hecke algebras}\label{eee}
Recall that $\BT=(\C^\times)^{n-1}$ stands for the  abstract Cartan torus of the group $SL_n$.
Let $\TV$ denote
the dual torus, so that
$\Hom(\TV,\C^\times)=\Hom(\C^\times,\BT)$
is a lattice in $\bt$.
The symmetric group $W=\syn$ acts naturally on $\BT,\ \TV$,
hence also on $\C[\TV]$, a  Laurent polynomial ring.
Let $\C[\TV/ W]=\C[\TV]^ W\sset\C[\TV] $ denote the subalgebra
of $ W$-invariants.

Associated with any $\Th\in\TV/W$, there is
a maximal ideal in  $\C[\TV]^W$.
We let $\BI_\Th\sset\C[\TV]$ denote the ideal
generated by that
maximal ideal of the subalgebra $\C[\TV]^W$.
The quotient,
$R_\Th:=\C[\TV]/\BI_\Th,$ called  `{\em coinvariant algebra}',  is a vector space
of dimension $n!$ that comes equipped with the regular representation of the group
$ W$.

Given a complex number $q\in\C^\times,$
let $\hec$ be the  affine Hecke algebra of type $\mathbf{A}_{n-1},$ 
modelled on  $ W\ltimes \Hom(\TV,\C^\times)$,
an affine Weyl group. Thus, there is a standard  commutative subalgebra
$\C[\TV]\sset \hec,$
the Bernstein subalgebra, such that 
the corresponding $ W$-invariants,
$\C[\TV]^ W\sset\C[\TV]\sset\hec,$ form the center
of  the  affine Hecke algebra.
There is a  natural  $\C[\TV]^ W$-linear action 
of the algebra $\hec$ on  $\C[\TV]$
via so-called Demazure-Lusztig operators, cf. eg. \cite{CG}, ch.~7.
In particular, for any  $\Th\in\TV/W$, the coinvariant algebra
$R_\Th=\C[\TV]/\BI_\Th$ inherits an  $\hec$-module structure.

Let $\treg\sset \BT=(\C^\times)^{n-1}$ be the
complement of the big diagonal, the subset of 
points with pairwise distinct coordinates.
Let  $\xreg\sset\X=SL_n\times\P,$ be an open subset formed
by the pairs $(g,\ell),$ such that $g\in G$, is a regular semisimple element,
and such that 
the line $\ell$ is {\em cyclic}
for $g$, i.e., such that we have $\C[g]\ell=\C^n$. 
We write $\supp: \xreg\onto\treg/ W,\ (g,\ell)\mto\Spec(g),$ 
for the map that assigns to a pair $(g,\ell)$
the unordered $n$-tuple of eigenvalues of the matrix $g$.

The fundamental group of the space $\treg/W$ is known to be  the  affine braid group $\baf$.
Thus, choosing a base point $x\in \treg/W$, for any $q\in\C^\times,$ we have a diagram
\beq{bafmap}
a:\ \pi_1(\treg/ W,x)=\baf\too\hec,
\eeq

Given a pair $(\Th,q)\in\TV/W\times\C^\times,$
let $a^*(R_\Th)$ be the pull-back of
the  $\hec$-module $R_\Th$ via the map
 \eqref{bafmap}.
Associated with  $a^*(R_\Th)$,
one has a local system on $\treg/ W$,
and we write  $\CR_{\Th,q}$ for  the pull-back of
the latter local system to $\xreg$ via the map $\supp$.

\subsection{Monodromy conjecture}
According to \cite{FG}, Proposition 3.2.3, there
is a canonical rational $G$-invariant section $\bff$, of the line bundle
 $\omega_\X^{\o 2},$ such that $\bff$ has neither zeros nor poles on 
the open set $\xreg$. 
Observe further that, for any $c\in\C$ and any differential operator
$u$ on $\xreg$, the formula $\bff^{-c}\ccirc u\ccirc \bff^c$ gives
a well defined  twisted differential operator
$\bff^{-c}\ccirc u\ccirc \bff^c\in \D^{2nc}_\X|_{\xreg}$. It follows that
the assignment $u\mto \bff^{-\frac{c}{2n}}\ccirc u\ccirc \bff^{\frac{c}{2n}}$
induces an isomorphism $\D_\X|_{\xreg}\iso \D^c_\X|_{\xreg}$, of  TDO.

We conclude that, given a $\D^c_\X$-module $\CM$, one may view 
$\CM|_{\xreg},$ the restriction of $\CM$ to the open set $\xreg$,
as a $\D_\xreg$-module via the above isomorphism of  TDO.

The  map
  $\bt^*\onto\bt^*/\bt^*_\Z=\TV$
induces  a canonical projection
$\bt^*/W\onto\TV/W$.

We call a complex number $c\in\C$ `{\em good}' if it is
not a negative rational number of the form $c=-p/m$ where
$2\leq m\leq n,\ 1\leq p\leq m,$ and $(p,m)=1$.

\begin{conjecture}\label{factor} Let $c\in\C$ be good.
 We  put $q:=\exp(2\pi ic)$,
and let $\Th\in \TV/W$ be the image of
an element $\th\in\bt^*/W$ under the canonical projection.

Then,  the locally constant sheaf
on $\xreg$ associated with  the $\D$-module $\CG^{\th,c}|_{\xreg}$
via the Riemann-Hilbert correspondence
is isomorphic to $\CR_{\Th,q}$.
\end{conjecture}

Let $\mathsf{H}_\kappa$ be the trigonometric Cherednik algebra
with parameter $\kappa=c/n$,
and let $e\mathsf{H}_\kappa e$ denote the
corresponding spherical subalgebra, cf. \cite[\S5]{FG}.
The condition that $c$ be good insures, by  \cite{BE},
that the algebras  $\mathsf{H}_\kappa$ and
$e\mathsf{H}_\kappa e$ are  Morita equivalent, cf. \cite[Proposition
3.1.3]{FG}.

Recall further that, according to \cite{GG} and \cite{FG},
the space $\BH(\CG^{\th,c}):=\Gamma(\X,\CG^{\th,c})^{\sln}$,
called the {\em Hamiltonian reduction} of the
$\D$-module $\CG^{\th,c}$,
has a natural $e\mathsf{H}_\kappa e$-module structure.
It is clear that describing the de Rham local
system of the $\D$-module
$\CG^{\la,c}|_{\xreg}$ amounts to
studying the monodromy of the corresponding
$\D$-module
$\D(\treg/ W)\bigotimes_{e\mathsf{H}_\kappa e}\BH(\CG^{\th,c}),$
on $\treg/ W.$

Assume that $c$ is good, so the above mentioned  Morita equivalence
holds. Then,
the latter problem is  equivalent
 to a similar
problem for the $\sH_\kappa$-module, $\CP_{\th,c}$,
that  corresponds to the $e\mathsf{H}_\kappa e$-module $\BH(\CG^{\th,c})$
via the  Morita equivalence. 

The $\sH_\kappa$-module $\CP_{\th,c}$ has been
computed in \cite{GG}, Lemma 7.5 in the rational case. In the
trigonometric
setting of the present section, the
corresponding result reads
\begin{equation}\label{CP}\CP_{\th,c}=
\hh_\kappa\otimes_{\C[\TV]\rtimes W} R_\Th,
\end{equation}
is the $\hh_\kappa$-module induced
from the representation $R_\Th$, of the subalgebra
$\C[\TV]\rtimes W.$ The corresponding local system
on  $\treg/ W$, comes from an $ W$-equivariant
$\D$-module,
$\KZ(\CP_{\th,c})$, on $\treg$,
cf. \cite{GGOR} for details about
the functor $\KZ$. The latter $\D$-module is nothing 
but the Knizhnik-Zamolodchikov connection in the
trivial vector bundle 
on $\treg$ associated with the representation $R_\Th;$
the formula for the connection can be found e.g. in
\cite{FV}, \S 3.1.
 
Thus, we conclude that our original
problem about the monodromy of the
Harish-Chandra $\D$-module $\CG^{\th,c}|_\xreg$
is equivalent to a similar problem 
for  the Knizhnik-Zamolodchikov connection in
the representation $R_\Th.$

For general enough $c\in\C$, the monodromy of
the Knizhnik-Zamolodchikov connection that arises from
\eqref{CP} has been studied in
\cite{Ch1}, Theorem 3.3; \cite{Ch2}, Theorem 3.6, and also
 \cite{Op}, Corollary ~6.9. The results in {\em loc cit}
confirm that our Conjecture \ref{factor} does hold
for sufficiently general parameters $c\in \C$.

\begin{remark} Let
$\dis \CP'_c=\hh_\kappa\otimes_{\C[\TV]\rtimes W} \C[ W],$
be an $\hh_\kappa$-module induced from the
regular representation of the group $ W$,
equipped with the `trivial' $\C[\TV]$-action, via the morphism $\C[\TV]\to\C,\, P\mto ~P(1)$.

We note that an analogue of our
 monodromy conjecture, with
  the  $\hh_\kappa$-module in \eqref{CP} being replaced
by the  $\hh_\kappa$-module $\CP'_c$, is known to be {\em false}, in general.
\end{remark}


\setcounter{equation}{0}
\vskip -10mm
\small{

}
\bigskip
\medskip
\footnotesize{\pb{{\bf M.F.}: IMU, IITP,
 and State University Higher School of Economics,
Dept. of Mathematics, 20 Myasnitskaya st.
Moscow 101000 Russia;\\
\hphantom{x}\enspace\en {\tt fnklberg@gmail.com}}}
\bigskip

\footnotesize{\pb{{\bf V.G.}: 
Department of Mathematics, University of Chicago, 
Chicago IL 60637, USA;\\
\hphantom{x}\en\en {\tt ginzburg@math.uchicago.edu}}}

\end{document}